\documentclass[11pt]{amsart}
\usepackage[dvipsnames]{xcolor}
\usepackage[normalem]{ulem}
\usepackage{amsmath,amsfonts,amssymb,euscript,mathrsfs,fullpage,blkarray,graphicx,multirow,mathtools,bm,ytableau,caption,subcaption,hyperref,tikz,cases}
\hypersetup{colorlinks, citecolor=red, filecolor=black, linkcolor=blue, urlcolor=blue}
\usetikzlibrary{patterns}

\newtheorem{thm}{Theorem}[section]

\theoremstyle{definition}
\newtheorem{defn}[thm]{Definition}

\newtheorem{ex}[thm]{Example}
\theoremstyle{remark}
\newtheorem{rem}[thm]{Remark}
\numberwithin{equation}{section}

\newcommand{\st}{\colon} 
\newcommand{\sm}{\setminus}
\DeclareMathOperator{\im}{im}
\DeclareMathOperator{\inv}{inv}

\DeclareMathOperator{\rk}{rk}

\DeclareMathOperator{\Sig}{Sig}

\newcommand{\abs}[1]{\rvert#1\lvert}
\newcommand{\dju}{\mathbin{\mathaccent\cdot\cup}}
\newcommand{\covers}{\gtrdot}
\newcommand{\coveredby}{\lessdot}
\newcommand{\B}{\mathcal{B}}
\newcommand{\Qq}{\mathbb{Q}}
\newcommand{\Nn}{\mathbb{N}}
\newcommand{\Rr}{\mathbb{R}}
\newcommand{\Zz}{\mathbb{Z}}
\newcommand{\shiftgen}[1]{\langle\!\langle#1\rangle\!\rangle} 
\newcommand{\Red}[1]{{\color{red}{#1}}}

\newcommand{\Blue}[1]{{\color{blue}{#1}}}

\newcommand{\Green}[1]{{\color{ForestGreen}{#1}}}

\newcounter{x}
\newcounter{y}
\newcounter{z}
\newcommand\xaxis{210}
\newcommand\yaxis{-30}
\newcommand\zaxis{90}
\newcommand\topside[3]{
  \fill[fill=blue, draw=black,thick,shift={(\xaxis:#1)},shift={(\yaxis:#2)},
  shift={(\zaxis:#3)}] (0,0) -- (30:1) -- (0,1) --(150:1)--(0,0);
  \fill[pattern=vertical lines, pattern color=black, draw=black,thick,shift={(\xaxis:#1)},shift={(\yaxis:#2)},
  shift={(\zaxis:#3)}] (0,0) -- (30:1) -- (0,1) --(150:1)--(0,0);
}
\newcommand\leftside[3]{
  \fill[fill=green, draw=black,thick,shift={(\xaxis:#1)},shift={(\yaxis:#2)},
  shift={(\zaxis:#3)}] (0,0) -- (0,-1) -- (210:1) --(150:1)--(0,0);
  \fill[pattern=north west lines, pattern color=black, draw=black,thick,shift={(\xaxis:#1)},shift={(\yaxis:#2)},
  shift={(\zaxis:#3)}] (0,0) -- (0,-1) -- (210:1) --(150:1)--(0,0);
}
\newcommand\rightside[3]{
  \fill[fill=red, draw=black,thick,shift={(\xaxis:#1)},shift={(\yaxis:#2)},
  shift={(\zaxis:#3)}] (0,0) -- (30:1) -- (-30:1) --(0,-1)--(0,0);
  \fill[pattern=grid, pattern color=black, draw=black,thick,shift={(\xaxis:#1)},shift={(\yaxis:#2)},
  shift={(\zaxis:#3)}] (0,0) -- (30:1) -- (-30:1) --(0,-1)--(0,0);
}
\newcommand\cube[3]{
  \topside{#1}{#2}{#3} \rightside{#1}{#2}{#3}
  \leftside{#1}{#2}{#3} 
}

\newcommand\planepartition[1]{
 \setcounter{x}{-1}
  \foreach \a in {#1} {
    \addtocounter{x}{1}
    \setcounter{y}{-1}
    \foreach \b in \a {
      \addtocounter{y}{1}
      \setcounter{z}{-1}
      \foreach \c in {1,...,\b} {
        \addtocounter{z}{1}
        \cube{\value{x}}{\value{y}}{\value{z}}
      }
    }
  }
}

\author[A. M. Duval]{Art M.\ Duval}
\address{Department of Mathematical Sciences\\ University of Texas at El Paso}
\email{aduval@utep.edu}
\thanks{AMD was supported in part by Simons Collaboration Grant \#516801.}
\author[W. Kook]{Woong Kook}
\address{Department of Mathematical Sciences, Seoul National University, Seoul, Republic of Korea.}
\email{woongkook@snu.ac.kr}
\thanks{WK and KL were supported in part 
by the National Research Foundation of Korea (NRF) Grants funded by the Korean Government (MSIP) (No.RS-2022-00165404).}
\author[K.-J. Lee]{Kang-Ju Lee}
\address{Research Institute of Mathematics, Seoul National University, Seoul, Republic of Korea.}
\email{leekj0706@snu.ac.kr}
\thanks{KL was supported in part by the National Research Foundation of Korea (NRF) Grants funded by the Korean Government (MSIP) (No.2021R1C1C2014185).}
\author[J. L. Martin]{Jeremy L.\ Martin}
\address{Department of Mathematics\\ University of Kansas}
\email{jlmartin@ku.edu}
\thanks{JLM was supported in part by Simons Collaboration Grant \#315347.}
\title[Simplicial resistance and spanning trees]{Simplicial effective resistance and enumeration of spanning trees}
\subjclass[2020]{Primary 
05E45; 
Secondary
05C05, 
05C50, 
31C20, 
94C15} 

\keywords{simplicial effective resistance, simplicial spanning tree, recursive structure, color-shifted complex, shifted complex}
\date{\today}
\begin{document}

\begin{abstract}
A graph can be regarded as an electrical network in which each edge is a resistor.  This point of view relates combinatorial quantities, such as the number of spanning trees, to electrical ones such as effective resistance.  The second and third authors have extended the combinatorics/electricity analogy to higher dimension and expressed the simplicial analogue of effective resistance as a ratio of weighted tree enumerators.  In this paper, we first use that ratio to prove a new enumeration formula for color-shifted complexes, confirming a conjecture by Aalipour and the first author, and generalizing a result of Ehrenborg and van Willigenburg on Ferrers graphs.  We then use the same technique  to recover an enumeration formula for shifted complexes, first proved by Klivans and the first and fourth authors.  In each case, we add facets one at a time, and give explicit expressions for simplicial effective resistances of added facets by constructing high-dimensional analogues of currents and voltages (respectively homological cycles and cohomological cocycles).
\end{abstract}
\maketitle

\section{Introduction}
This paper is about counting spanning trees in simplicial complexes by viewing them as higher-dimensional electrical networks.

The idea of representing electrical networks as graphs can be traced back to the foundational work of Kirchhoff~\cite{Kir}.  The reverse point of view, using the electrical model to study the combinatorics of a graph, is more recent; see, e.g., \cite{Biggs,BB}.
Kirchhoff's fundamental laws on current and voltages are naturally expressed in the algebraic graph theory language of flows and cuts.
We regard a graph as a network in which each edge has unit resistance, then attach a new edge $e$ carrying a specified amount of current.  The resulting currents and potentials must satisfy Ohm's and Kirchhoff's network laws, and can be determined via the (discrete) Dirichlet principle: the system will arrange itself so as to minimize energy.  A key quantity of this model, the \textit{effective resistance} between the endpoints of~$e$, has explicit combinatorial meaning: it equals the fraction of spanning trees of $G+e$ that contain~$e$ \cite[Prop.~17.1]{Biggs}.  For graph families with a recursive structure, this calculation can serve as the inductive step in giving a closed formula for the number of spanning trees.  Assigning an indeterminate resistance $r_e$ to each edge $e$ leads to analogous formulas for \textit{weighted} spanning tree enumerators, in which each edge is weighted by $1/r_e$.

Here we consider generalizations of these ideas from graphs to simplicial complexes.  The theory of simplicial spanning trees was pioneered by Kalai~\cite{K} and developed subsequently by Adin~\cite{A}, Petersson~\cite{Petersson}, Lyons~\cite{Ly}, Catanzaro--Chernyak--Klein \cite{CCK1,CCK2}, and Klivans and the first and fourth authors~\cite{DKM1,DKM2}; for a recent overview, see~\cite{DKM}. 
Meanwhile, the authors in \cite{CCK1} initiated the study of high-dimensional electrical networks, extending Kirchhoff's current and voltage laws via algebraic topology.
The second and third authors~\cite{KL1} introduced effective resistance for simplicial complexes and showed that it carries the same combinatorial interpretation as in the graph case.
We give the necessary background in the first part of the paper.  Section~2 reviews standard facts about simplicial complexes and homology; Section 3 concerns simplicial spanning trees and their enumeration; and Section 4 develops the theory of simplicial networks, including its connection to higher-dimensional spanning trees.

In the main part of the paper, Sections 5 and 6, we apply the theory of~\cite{KL1} to give explicit enumeration formulas
for spanning trees of \textit{shifted} and \textit{color-shifted} simplicial complexes, which we will define shortly.  
Most previous computations of tree enumeration of higher-dimensional complexes, including standard simplices \cite{K}, complete colorful \cite{A}, shifted \cite{DKM1}, cubical \cite{DKM2}, and matroid \cite{KL} complexes,
relied on first computing the Laplacian eigenvalues of the complexes, and then applying some version of the Matrix-Tree Theorem~\cite[Secs.~3.4-3.5]{DKM}.  Color-shifted complexes are not amenable to using the Matrix-Tree Theorem this way, because their Laplacian eigenvalues are not nice; even in the unweighted case, the eigenvalues are generally not integers.

To define color-shifted complexes, we regard the numbers $q=1,\dots,d+1$ as colors, and let $V_q=\{v_{q,1},\dots,v_{q,n_q}\}$ be a linearly ordered set of vertices of color $q$.  A pure $(d+1)$-dimensional complex on $V=V_1\cup\cdots\cup V_{d+1}$ is \textit{balanced} if each facet contains one vertex of each color; see \cite[Sec.~III.4]{CCA}.  A balanced simplicial complex $\Delta$ is \textit{color-shifted} if any vertex of a face may be replaced with a smaller vertex of the same color to obtain a new face.  Equivalently, the facets form an order ideal in the Cartesian product $V_1\times\cdots\times V_{d+1}$.
The canonical spanning tree is the subcomplex of faces having at least one vertex that is minimal in its color class.
Color-shifted complexes were introduced by Babson and Novik \cite{BN} and generally behave like multipartite analogues of shifted complexes.  A color-shifted complex of dimension~1 is known as a \textit{Ferrers graph}, since it can be described by a Ferrers diagram in which rows are red vertices, columns are blue vertices, and squares are edges.  Ehrenborg and van~Willigenburg \cite[Prop.~2.2]{EW} used electrical graph theory to count spanning trees of Ferrers graphs.  G.~Aalipour and the first author conjectured a more general formula~\cite[eqn.~28]{DKM} and proved the 2-dimensional case using classical methods.  Here we use simplicial effective resistance to prove the general version, which appears as Theorem~\ref{thm:CSC-enumeration}. 

Adin~\cite{A} counted the spanning trees of the \textit{complete colorful complex}, whose facets are all the $(d+1)$-sets with one vertex of each color; his formula confirmed a conjecture of Bolker~\cite[eqn.~20]{Bolker}.  Building on Adin's methods, Aalipour and the present authors~\cite[Thm.~1.2]{ADKLM} computed the corresponding vertex-weighted enumerator.  In principle, one can enumerate spanning trees of any color-shifted complex~$\Delta$ (indeed, any balanced complex) by setting all the weights of nonfaces to zero, i.e., working in the Stanley-Reisner ring of $\Delta$, although it is unclear whether this method will produce formulas that one can write down.  On the other hand, these earlier formulas can be obtained as a special case of Theorem~\ref{thm:CSC-enumeration} (since complete colorful complexes are certainly color-shifted), as we explain in Example~\ref{ex:ccc}.

A (pure) simplicial complex $\Delta$ on vertex set $[n]=\{1,\dots,n\}$ is called \textit{shifted} if any vertex of a face may be replaced with a smaller vertex to produce another face.  That is, if $\sigma\in\Delta$ and $1\leq i<j\leq n$ with $j\in\sigma$ and $i\not\in\sigma$, then $\sigma\setminus j\cup i\in\Delta$.
Equivalently, define the Gale (partial) order $\leq$ on $k$-subsets of $[n]$ as follows: if $a=\{a_1<\ldots<a_k\}$ and $b=\{b_1<\ldots<b_k\}$, then $a\leq b$ iff if $a_i \leq b_i$ for all $1 \leq i \leq k$.  Then a complex is shifted if its facets form an order ideal with respect to Gale order.  The canonical spanning tree is the star of vertex~1.  
Shifted families of sets were introduced (though not with that name) by Erd\"os, Ko, and Rado~\cite{EKR} to prove results in extremal combinatorics, where they have since been used extensively (see for instance~\cite{Frankl}).  Bj\"orner and Kalai~\cite{BK} characterized the $f$-vectors and Betti sequences of simplicial complexes using algebraic shifting, which was introduced by Kalai~\cite{K-f} and which produces a shifted complex with some of the same features as the original complex.  Shifted complexes enjoy numerous other good properties:  
for example, they correspond algebraically to Borel-fixed monomial ideals~\cite{Shifting} and
have integer Laplacian eigenvalues~\cite{DR}.  In dimension~1 they specialize to threshold graphs.  Weighted tree enumerators of shifted complexes were studied in~\cite{DKM1}.  Here we use simplicial effective resistance to obtain the weighted formula \cite[eqn.~25]{DKM}, which we restate as Theorem~\ref{thm:shifted}.
 
To enumerate the trees of a color-shifted or shifted complex $\Delta$, we construct $\Delta$ inductively by starting with a canonical spanning tree and adding facets one at a time so that every subcomplex obtained along the way is color-shifted or shifted, respectively.  The key step is computing explicit expressions for simplicial effective resistances of added facets, which appear as Theorems~\ref{thm:CSC-ratio} and ~\ref{thm:SC-ratio}, respectively.  We obtain these expressions by constructing high-dimensional currents (homological cycles) and voltages (cohomological cycles).
This technique can be replicated for other families $\mathscr{F}$ of simplicial (or even cellular) complexes given the right ingredients:
First, all complexes in $\mathscr{F}$ share a canonical spanning tree $T$.
Second, we need a partial order on all facets of complexes in $\mathscr{F}$, such that  (i) $T$ is an order ideal and (ii) the boundary of each facet not in $T$ lies in the complex generated by all smaller facets.
Third, for each complex $\Delta\in\mathscr{F}$ and each $\leq$-minimal facet $\sigma\not\in\Delta$, we need an expression for the effective resistance of $\sigma$ with respect to $\Delta$ that is independent of $\Delta$ itself.
If these conditions are met, then any linear extension of the partial order of facets will allow the inductive computation of tree enumeration by simplicial effective resistance.

\textbf{Acknowledgments.} We are grateful to an anonymous referee for a very careful reading of the article, resulting in improved exposition throughout and a simplification of one of the technical arguments in Theorem~\ref{thm:SC-ratio}.

\section{Preliminaries} \label{sec:prelim}
Let $[n]$ denote $\{1,2,\dots,n\}$ for a natural number $n$ and $[a,b]$ denote $\{a,a+1,\dots,b-1,b\}$ for integers $a,b$. For a set $\tau$ and an element $j$, we will use the abbreviated notations $\tau \cup j$ and $\tau \setminus j$ for $\tau \cup \{j\}$ and $\tau \setminus \{j\}$, respectively. We refer to standard texts (e.g., \cite{Mu,Hatcher}) for basic definitions and background in algebraic topology. 

\subsection{Abstract simplicial complexes}
For a finite set $V$, a non-empty collection $\Delta$ of subsets of $V$ is called an (abstract) \emph{simplicial complex} if $\Delta$ is closed under inclusion, i.e., if $\tau \in \Delta$ and $\tau' \subseteq \tau$, then $\tau' \in \Delta$. The set $V$ is called the \emph{vertex set} of $\Delta$; we will assume that the elements of $V$ are ordered. The elements of $\Delta$ are called \emph{faces}; a maximal face is called a \emph{facet}.  The dimension of a face $\tau \in \Delta$ is $\dim \tau = |\tau|-1$, and the dimension of $\Delta$ is the maximum dimension of the elements in $\Delta$.  The complex $\Delta$ is \emph{pure} if all facets have the same dimension.  The collection of all $i$-dimensional faces of $\Delta$ is denoted by $\Delta_i$; in particular $\Delta_{-1}=\{\emptyset\}$. The $i$th \emph{skeleton} of $\Delta$ is the subcomplex $\Delta^{(i)}=\bigcup_{-1\leq j \leq i} \Delta_{j}$.  We will sometimes allow two facets to share the same set of proper subfaces, so that we are really working in the slightly larger class of $\Delta$-complexes \cite[Sec.~2.1]{Hatcher}.

\subsection{Boundary operators and chain groups} \label{sec:chain}
For an $(i+1)$-set $\tau=\{v_{j_1},v_{j_2},\dots,v_{j_{i+1}}\}$ with $j_1< j_2 <\dots<j_{i+1}$, the \emph{oriented simplex} $[\tau]$ is defined to be
\[
[\tau]=[v_{j_1},v_{j_2},\dots,v_{j_{i+1}}]=(-1)^{\inv(\pi)}[v_{j_{\pi(1)}},v_{j_{\pi(2)}},\dots,v_{j_{\pi(i+1)}}],
\]
where $\pi$ is a permutation on $[i+1]$, and $\inv(\pi)$ is the number of inversions in $\pi$. Define the \emph{boundary operator} $\partial_{i}$ by 
\[
\partial_{i}[\tau]=\sum_{\hat{\imath}=1}^{i+1}(-1)^{\hat{\imath}-1}[\tau\setminus v_{j_{\hat{\imath}}}].
\] 
Then $\partial_{i}\partial_{i+1}=0$ for each natural number $i$. 

Let $R$ be a commutative ring.  The $i$th \emph{chain group of $\Delta$ with coefficients in $R$} is the free $R$-module $\mathcal{C}_i(\Delta)=\mathcal{C}_i(\Delta;R)\cong R^{|\Delta_i|}$ with basis the oriented simplices $[\tau]$ for $\tau\in \Delta_i$.
An element $c \in \mathcal{C}_i(\Delta)$ can be written as 
\[
c=\sum_{\tau \in \Delta_i}c_{\tau}[\tau],
\]
and this element will be regarded as a vector $(c_{\tau})_{\tau \in \Delta_i}$ of length $|\Delta_i|$.  Define an inner product $\langle\,,\,\rangle$ on $\mathcal{C}_{i}(\Delta)$ by declaring the set of oriented simplices to be an orthonormal basis: that is, $\langle c,c'\rangle=\sum_{\tau\in \Delta_i} c_\tau c'_\tau$.

The $i$th \emph{boundary operator} $\partial_{\Delta,i}=\partial^R_{\Delta,i}\colon \mathcal{C}_{i}(\Delta;R) \to \mathcal{C}_{i-1}(\Delta;R)$ is defined by $\partial_{\Delta,i}[\sigma]=\partial_{i}[\sigma]$ for each $i$-face $\tau \in \Delta_d$.  In particular, the $0$th boundary operator $\partial_{\Delta,0}\colon \mathcal{C}_{0}(\Delta)\to \mathcal{C}_{-1}(\Delta) \cong R$ is defined by $\partial_{\Delta,0}[v]=1$ for each $v\in \Delta_{0}$. 
The boundary operator $\partial_{\Delta,i}$ may be regarded as an $|\Delta_{i-1}| \times |\Delta_i|$ matrix with respect to the standard bases.

The elements of $\im\partial_{\Delta,i+1}$ are called $i$-\emph{boundaries} of $\Delta$, the elements of $\ker \partial_{\Delta,i}$ are $i$-\emph{cycles}, and the elements of the orthogonal complement $(\ker \partial_{\Delta,i})^\perp$ are $i$-\emph{cocycles}.

\subsection{Homology groups and torsion subgroups}

For $0\leq i\leq d$, the $i$th \emph{(reduced) homology of $\Delta$ with coefficients in $R$} is $\tilde{H}_{i}(\Delta;R) = \ker\partial^R_{\Delta,i}\, / \im\partial^R_{\Delta,i+1}$. In the important case $R=\Zz$, the homology $\tilde{H}_{i}(\Delta;\Zz)$ is a finitely generated abelian group, so it can be decomposed as 
\[
\tilde{H}_{i}(\Delta;\Zz)\equiv \Zz^{b_i(\Delta)} \oplus \mathbf{T}(\tilde{H}_{i}(\Delta,\Zz)),
\]
where $b_i(\Delta)=\rk\tilde{H}_{i}(\Delta)$ is the $i$th (topological) \emph{Betti number} of $\Delta$, and $\mathbf{T}(\cdot)$ is the \emph{torsion subgroup}. For simplicity, we define
\[
\mathbf{t}_i(\Delta)=|\mathbf{T}(\tilde{H}_{i}(\Delta;\Zz))|.
\]
A $d$-complex $\Delta$ such that $\tilde{H}_i(\Delta;R)=0$ for $i<d$ is called $R$-\textit{acyclic in positive codimension} (for short, $R$-APC).  Note that being $\Rr$-APC is equivalent to the condition that all Betti numbers vanish; this is a weaker condition than being $\Zz$-APC.  

\section{Simplicial spanning trees and their enumeration} \label{sec:SST}
\subsection{Simplicial spanning trees and tree-numbers}

We begin by reviewing the theory of \emph{simplicial spanning trees}, pioneered by Kalai \cite{K}.  Let $\Delta$ be a $d$-dimensional simplicial complex on vertex set $V$; we may assume $\Delta$ is pure. For $i \in[0, d]$, an
$i$-\emph{dimensional spanning tree} (or simply $i$-\emph{tree}) is a subcomplex $\Upsilon$ of dimension~$i$ such that
\begin{equation} \label{conditions-for-tree}
\Upsilon^{(i-1)}=\Delta^{(i-1)}, \quad b_{i-1} (\Upsilon)=b_{i-1} (\Delta),\quad \text{and}\quad b_{i}(\Upsilon)=0.
\end{equation}
In fact any two of these conditions together imply the third.  Let $\mathcal{T}_{i}(\Delta)$ be the collection of all $i$-trees $\Upsilon$ of $\Delta$. The $i$th \emph{tree-number} $k_{i}(\Delta)$ of $\Delta$ is defined to be
\[
k_{i}(\Delta)=\sum_{\Upsilon\in \mathcal{T}_{i}(\Delta)}{|\mathbf{T}(\tilde{H}_{i-1}(\Upsilon))|^2}=\sum_{\Upsilon\in \mathcal{T}_{i}(\Delta)}{\mathbf{t}_{i-1}(\Upsilon)^2}.
\]
Note that if $\Delta=G$ is a connected graph and $i=1$, then $k_{1}(G)$ is equal to the number of spanning trees of $G$.

Now assign a weight $x_\tau$ to each $d$-face $\tau \in \Delta_d$, and let $x_\Upsilon=\prod_{\tau \in \Upsilon_d}x_\tau$.  The \emph{weighted tree-number} $\hat{k}_d(\Delta)$ of $\Delta$ with these weights is
\begin{equation}\label{eq:WTN}
\hat{k}_d(\Delta)=
\sum_{\Upsilon\in\mathcal{T}_d(\Delta)}x_{\Upsilon}\,{\mathbf{t}_{d-1}(\Upsilon)^2}.
\end{equation}
If we introduce vertex weights $x_v$ and set $x_{\tau}=\prod_{v \in \tau}x_v$, then~\eqref{eq:WTN} becomes
\[
\hat{k}_d(\Delta)=
\sum_{\Upsilon\in\mathcal{T}_d(\Delta)}\prod_{\tau \in \Upsilon_d}\left(\prod_{v \in \tau} x_{\tau}\right){\mathbf{t}_{d-1}(\Upsilon)^2}=\sum_{\Upsilon\in\mathcal{T}_d(\Delta)}\left(\prod_{v \in V}x_{v}^{\deg_\Upsilon v}\right){\mathbf{t}_{d-1}(\Upsilon)^2}.
\]

\begin{ex}\label{eq:octa}
Let $\Delta$ be the octahedron with vertices $v_{q,j_q}$ for $q=1,2,3$ and $j_q=1,2$, and facets $v_{1,j_1}v_{2,j_2}v_{3,j_3}$ for $j_1,j_2,j_3=1,2$.  (See Figure~\ref{Fig:Octa}.)  The 2-dimensional spanning trees of $\Delta$ are the subcomplexes obtained by deleting a single facet.  Every such tree is contractible, hence torsion-free.  (This description of spanning trees is characteristic of a simplicial sphere.) Letting $x_{q,j_q}$ be the weight of vertex $v_{q,j_q}$, the $2$-dimensional weighted tree-number is
\begin{align*}
\hat{k}_2({\Delta})&=
\sum_{j_1,j_2,j_3=1,2}\dfrac{(x_{1,1}x_{1,2}x_{2,1}x_{2,2}x_{3,1}x_{3,2})^4}{x_{1,j_1}x_{2,j_2}x_{3,j_3}}
\\
&=
(x_{1,1}x_{1,2}x_{2,1}x_{2,2}x_{3,1}x_{3,2})^3(x_{1,1}+x_{1,2})(x_{2,1}+x_{2,2})(x_{3,1}+x_{3,2}).
\end{align*}
\end{ex}

\begin{figure}[ht]
\begin{center}

\begin{tikzpicture}[thick,scale=4.5]
\coordinate [label=left:$v_{2,2}$] (A1) at (0,0);
\coordinate [label=right:$v_{1,2}$] (A2) at (0.6,0.2);
\coordinate [label=right:$v_{2,1}$] (A3) at (1,0);
\coordinate [label=below:$v_{1,1}$] (A4) at (0.4,-0.2);
\coordinate [label=above:$v_{3,1}$] (B1) at (0.5,0.5);
\coordinate [label=below:$v_{3,2}$] (B2) at (0.5,-0.5);
	
\begin{scope}[thick,dashed,,opacity=0.6]
\draw (A1) -- (A2) -- (A3);
\draw (B1) -- (A2) -- (B2);
\end{scope}
\draw[fill=black,opacity=0.2] (A1) -- (A4) -- (B1);
\draw[fill=black,opacity=0.2] (A1) -- (A4) -- (B2);
\draw[fill=black,opacity=0.2] (A3) -- (A4) -- (B1);
\draw[fill=black,opacity=0.2] (A3) -- (A4) -- (B2);
\draw (B1) -- (A1) -- (B2) -- (A3) --cycle;
\end{tikzpicture}
\caption{The octahedron of Example~\ref{eq:octa}.\label{Fig:Octa}}
\end{center}
\end{figure}
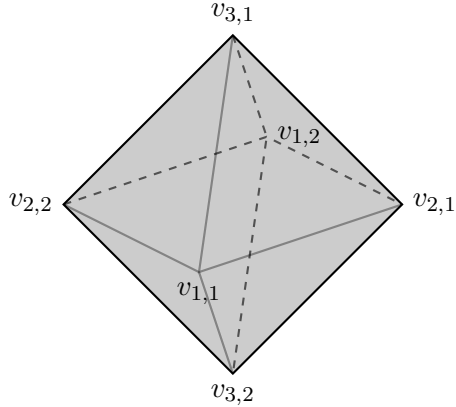

\subsection{Adding a facet}\label{sec:ET}
Let $\Psi$ be a $d$-dimensional simplicial complex and let $\sigma$ be a $(d+1)$-set that is not a face of $\Psi$ (in fact, the vertices of $\sigma$ are not required to be vertices of $\Psi)$.  Let $\Delta$ be the simplicial complex obtained from $\Psi$ by attaching $\sigma$ as a new simplex, i.e., $\Delta=\Psi\cup 2^{\sigma}$ where $2^{\sigma}$ is the power set of $\sigma$. In particular, $\Delta_d=\Psi_d\cup \{\sigma\}$.  Assign $\sigma$ the weight $x_{\sigma}$.
For the purpose of enumerating simplicial spanning trees, we are interested in the ratio of tree-numbers
$\hat{k}_d(\Delta)/\hat{k}_d(\Psi)$.

This tree-number ratio is easy to determine when $\Psi$ does not contain all proper subsets of~$\sigma$.  In this case, the map $\mathcal{T}_d (\Psi)\to\mathcal{T}_d(\Delta)$ defined by $\Upsilon \mapsto \Upsilon \cup \sigma$ is a torsion-preserving bijection. Therefore,
\begin{equation}\label{eq:rationotcg}
\dfrac{\hat{k}_d(\Delta)}{\hat{k}_d(\Psi)}=x_\sigma.
\end{equation}
In particular, if $\Psi$ is a connected graph (so $d=1$) and $\sigma$ is an edge with exactly one vertex in $\Psi$, then the spanning trees of $\Delta$ are precisely the graphs obtained by adjoining the edge $\sigma$ to a spanning tree of~$\Psi$.

When all proper subsets of $\sigma$ are contained in $\Psi$, the ratio $\hat{k}_d(\Delta)\,/\,\hat{k}_d(\Psi)$ is related to the combinatorial interpretation of its \emph{simplicial effective resistance}, defined in~\eqref{def:sef}.  By calculating the ratio $\hat{k}_d(\Delta)/\hat{k}_d(\Psi)$ and applying induction, we will derive the weighted tree-numbers of color-shifted complexes (Section~\ref{sec:color-shifted}) and shifted complexes (Section~\ref{sec:shifted}).

\begin{ex}\label{ex:octa_ratio} If $\Delta$ is the weighted octahedron given in Example~\ref{eq:octa} and $\Psi$ is the subcomplex obtained by deleting the facet $\sigma=v_{1,2}v_{2,2}v_{3,2}$, then
\begin{align*}
\dfrac{\hat{k}_2(\Delta)}{\hat{k}_2(\Psi)}
&= \dfrac{(x_{1,1}x_{1,2}x_{2,1}x_{2,2}x_{3,1}x_{3,2})^3(x_{1,1}+x_{1,2})(x_{2,1}+x_{2,2})(x_{3,1}+x_{3,2})}{(x_{1,1}x_{1,2}x_{2,1}x_{2,2}x_{3,1}x_{3,2})^4/(x_{1,2}x_{2,2}x_{3,2})}
\\
&= \dfrac{(x_{1,1}+x_{1,2})(x_{2,1}+x_{2,2})(x_{3,1}+x_{3,2})}{x_{1,1}x_{2,1}x_{3,1}}.
\end{align*}
\end{ex}

\section{Simplicial effective resistance and spanning trees} \label{sec:SER}

We now describe the electrical theory of simplicial networks~\cite{CCK1,KL1}, which is a high-dimensional generalization of the classical theory in dimension~1, introduced by Kirchhoff \cite{Kir} (see also \cite{Biggs} and \cite[Chap.~II]{BB}).  We will review simplicial effective resistance for simplicial networks and its combinatorial interpretation in terms of high-dimensional spanning trees.

\subsection{Simplicial networks and effective resistance} \label{subsec:SNER}
For a simplicial complex $\Delta$ of dimension~$d$, let $\mathcal{R}_\Delta=\{r_\tau >0\st\tau \in \Delta_d\}$ be a list of positive \emph{resistances} associated with the $d$-dimensional simplices of $\Delta$.  The pair $(\Delta,\mathcal{R}_\Delta)$ is called a \emph{simplicial resistor network}.  Let $\sigma$ be a $(d+1)$-subset of the vertex set of $\Delta$ that satisfies 
\begin{equation}\label{eq:bdcondition}
\partial_d[\sigma] \in \im\partial_{\Delta,d}
\end{equation}
(equivalently, attaching $\sigma$ to $\Delta$ increases the $d$th Betti number by 1).
Condition~\eqref{eq:bdcondition} will be necessary for the definition of effective resistance of~$\sigma$.  Note that when $d=1$, \eqref{eq:bdcondition} says that the endpoints of the edge $\sigma$ belong to the same connected component.

Consider a simplicial network $(\Delta,\mathcal{R}_\Delta)$ in which a current $\alpha$ ($>0$) is preassigned to each oriented face in $\partial_d[{\sigma}]$.
We now attach a simplex ${\bm\sigma}$ having vertex set $\sigma$.  We regard $\bm\sigma$ as a current generator supplying current $i_{\bm\sigma}=\alpha$. The complex $\Delta^{\bm\sigma}$ thus obtained can be regarded as a high-dimensional analogue of a resistor network together with a battery.\footnote{A real battery provides a fixed voltage rather than a fixed current, so this analogy is not physically accurate, but it is still useful conceptually.}  Denoting the voltage of ${\bm \sigma}$ by $v_{\bm \sigma}$, the ensuing state of the electrical network on $\Delta^{\bm\sigma}$ can be described by current and voltage vectors
$I_{\bm\sigma}=(i_\tau)_{\tau \in \Delta^{\bm\sigma}_d}$ and $V_{\bm\sigma}=(v_\tau)_{\tau \in \Delta^{\bm\sigma}_d}$, satisfying the following fundamental network laws:
\begin{align}
I_{\bm\sigma}&\in \ker \partial_{\Delta^{\bm\sigma}, d}
	&& \text{(Kirchhoff's current law (KCL))}, \label{KCL}\\
V_{\bm\sigma}&\in (\ker \partial_{\Delta^{\bm\sigma}, d})^{\perp}
	&& \text{(Kirchhoff's voltage law (KVL))},\label{KVL}\\
v_\tau&=r_\tau i_\tau \text{ for all } \tau \in \Delta_d
	&& \text{(Ohm's law (OL))}.\label{OL}
\end{align}

That is, Kirchhoff's laws say that the current ${I_{\bm\sigma}}$ is a $d$-cycle and that the voltage $V_{\bm\sigma}$ is a $d$-cocycle. The \textit{simplicial effective resistance} of the \emph{set} $\sigma$ \cite[Thm.~3.1]{KL1} is then defined to be
\begin{equation} \label{def:sef}
R_{\sigma}=-\dfrac{v_{\bm\sigma}}{i_{\bm\sigma}}.
\end{equation}

\begin{rem} \label{rem:contains-sigma}
Let us consider the case that $\Delta$ already contains $\sigma$ as a facet (before attaching $\bm\sigma$), so that $\Delta$ is a $\Delta$-complex in the sense of \cite{Hatcher} rather than a simplicial complex.  Then condition~\eqref{eq:bdcondition} is automatically satisfied.  Moreover, if we write $V=V_{\bm\sigma}-v_{\bm\sigma}[\bm\sigma]$ and $I=I_{\bm\sigma}-i_{\bm\sigma}[\bm\sigma]$, then $V_{\bm\sigma}\in (\ker \partial_{\Delta^{\bm\sigma}, d})^{\perp}$ ~\eqref{KVL} if and only if $V\in (\ker \partial_{\Delta, d})^{\perp}$ and $v_{\bm \sigma}=v_{\sigma}$. 
By OL~\eqref{OL}, $v_{\bm \sigma}=v_{\sigma}=i_{\sigma}r_{\sigma}$, so $R_{\sigma}$ can be written as
\begin{equation} \label{def:resef}
R_{\sigma}=-\dfrac{i_{\sigma}r_{\sigma}}{i_{\bm\sigma}}=-\dfrac{i_{\sigma}r_{\sigma}}{\alpha}.
\end{equation}
\end{rem}

\begin{ex}\label{ex:octa_elec}
We give an example of solving a simplicial network problem, that is, determining the current and voltage vectors produced by attaching a current generator to a simplicial resistor network, and computing effective resistance.  Let $\Delta$ be the octahedron of Example~\ref{eq:octa}, with resistances given by $r_{\tau}=1/(x_{1,j_1} x_{2,j_2}x_{3,j_3})$ for each facet $\tau =v_{1,j_1}v_{2,j_2}v_{3,j_3}\in \Delta$.  For convenience, set
\begin{align*}
x_{j_1j_2j_3} &= x_{1,j_1}x_{2,j_2}x_{3,j_3} \text{ for } j_1,j_2,j_3\in\{1,2\},\\
x_{222}^* &= (x_{1,1}+x_{1,2})(x_{2,1}+x_{2,2})(x_{3,1}+x_{3,2}).
\end{align*}
Let us attach a current generator ${\bm\sigma}=v_{1,2}v_{2,2}v_{3,2}$ with current $i_{\bm\sigma}=x_{222}^*$. 
We will write chains in $\mathcal{C}_2(\Delta^\sigma)$ as vectors whose entries correspond to the 2-faces of $\Delta$ in lexicographic order, followed by ${\bm\sigma}$ as the last entry. Let $z_{\sigma}$ be a generator of $\tilde{H}_2(\Delta)$ and let $z_{\bm \sigma}=[\bm{\sigma}]-[\sigma]$, which we can write as vectors 
\[
z_{\sigma}=[-1,\:1,\:1,\:-1,\:1,\:-1,\:-1,\:1,\:0]^{tr} 
\mbox{  and  }
z_{\bm\sigma}=[0,\:0,\:0,\:0,\:0,\:0,\:0,-1,\:1]^{tr}.
\]
The cycle space $\ker \partial_{\Delta^{\bm\sigma}, 2}$ is generated by $z_\sigma$ and $z_{\bm \sigma}$, so the current $I_{\bm\sigma}$ that takes the given value on ${\bm\sigma}$ and that satisfies KCL is
\[I_{\bm\sigma} =
 a\: z_{\sigma}+ x_{222}^* \: z_{\bm \sigma}
=
[-a,\:a,\:a,\:-a,\:a,\:-a,\:-a,\:a-x_{222}^*,\:x_{222}^*]^{tr}
\]
for some $a\in \mathbb{R}$.  Note that $I_{\bm\sigma}$ is unique up to the choice of $a$.
The voltages for faces of $\Delta$ can be found by applying Ohm's law, and the voltage for ${\bm\sigma}$ by KVL, giving
\[V_{\bm\sigma} = \Big[-\dfrac{a}{x_{111}},\:\dfrac{a}{x_{112}},\:\dfrac{a}{x_{121}},\:-\dfrac{a}{x_{122}},\:\dfrac{a}{x_{211}},\:-\dfrac{a}{x_{212}},\:-\dfrac{a}{x_{221}},\:\dfrac{a-x_{222}^*}{x_{222}},\:\dfrac{a-x_{222}^*}{x_{222}}\Big]^{tr}.
\]
Since the voltage $V_{\bm\sigma}$ satisfies KVL, 
\[
0=\left\langle V_{\bm\sigma},z_\sigma\right\rangle
= \dfrac{a \, x_{222}^*}{x_{111}x_{222}}-\dfrac{x_{222}^*}{x_{222}}, \mbox{ so that } {a=x_{111}}.
\]
Consequently, 
\begin{align*}
I_{\bm\sigma} 
&=
[-x_{111},\:x_{111},\:x_{111},\:-x_{111},\:x_{111},\:-x_{111},\:-x_{111},\:x_{111}-x_{222}^*,\:x_{222}^*]^{tr}\\
V_{\bm\sigma} &= \Big[-\dfrac{x_{111}}{x_{111}},\:\dfrac{x_{111}}{x_{112}},\:\dfrac{x_{111}}{x_{121}},\:-\dfrac{x_{111}}{x_{122}},\:\dfrac{x_{111}}{x_{211}},\:-\dfrac{x_{111}}{x_{212}},\:-\dfrac{x_{111}}{x_{221}},\:\dfrac{x_{111}-x_{222}^*}{x_{222}},\:\dfrac{x_{111}-x_{222}^*}{x_{222}}\Big]^{tr}.
\end{align*}
The simplicial effective resistance $R_{\sigma}$ is
\[R_{\sigma}=-\dfrac{v_{\bm\sigma}}{i_{\bm\sigma}}=\dfrac{x_{222}^*-x_{111}}{x_{222}}\dfrac{1}{x_{222}^*}=\dfrac{1}{x_{222}}\Big(\dfrac{x_{222}^*-x_{111}}{x_{222}^*}\Big).
\]
\end{ex}

\subsection{Simplicial effective resistance and spanning trees}
We now view $\Delta^{\bm \sigma}$ as a weighted simplicial complex in which each $d$-face $\tau\in\Delta$ is weighted by its \emph{conductance} $x_{\tau}=1/r_{\tau}$, while the $d$-face ${\bm\sigma}$ in $\Delta^{\bm \sigma}$ has weight $1$. For a $d$-dimensional spanning tree $\Upsilon$ of $\Delta^{\bm \sigma}$, define its weight $x_\Upsilon$ to be 
\[
x_\Upsilon=\prod_{\tau \in \Upsilon_d}x_\tau.
\]
Define $\mathcal{T}_d(\Delta)_{\bm \sigma}=\{\Upsilon \in \mathcal{T}_d(\Delta^{\bm \sigma}) \st{\bm\sigma} \in \Upsilon_d \}$, and 
\[
\hat{k}_d(\Delta)_{\bm \sigma}=
\sum_{\Upsilon\in\mathcal{T}_d(\Delta)_{\bm \sigma}}x_{\Upsilon}\,{\mathbf{t}_{d-1}(\Upsilon)}^2.
\]
The following theorem gives the combinatorial interpretation of $R_{\sigma}$ as a ratio of weighted tree numbers.  The corresponding result for graphs is \cite[Prop.~2.3]{T} and \cite[Thm.~6]{K3}, and the unweighted case is \cite[Prop.~17.1]{Biggs}.

\begin{thm}[{\cite[Thm.~5.1]{KL1} and \cite[Thm.~6]{KL2}}] \label{thm:fortree}
The simplicial effective resistance $R_{\sigma}$ in $(\Delta,\mathcal{R}_\Delta)$ is given by
\begin{equation}\label{eq:SER}
R_{\sigma}=\dfrac{\hat{k}_d(\Delta)_{\bm \sigma}}{\hat{k}_d(\Delta)}.
\end{equation}
\end{thm}

This combinatorial interpretation for $R_{\sigma}$ will be employed to give formulas for the tree-numbers of color-shifted complexes and shifted complexes. To this end, we will rewrite Equation~\eqref{eq:SER} as follows. Let $\Psi$ be a $d$-dimensional simplicial complex, and let $\sigma$ be a $d$-simplex not in $\Psi$, but whose proper subsets are all faces of~$\Psi$. Let $\Delta=\Psi \cup \{\sigma\}$. Since $\bm \sigma$ has the same vertex set as $\sigma$, the map $\{\Upsilon \in \mathcal{T}_d(\Delta) \st{\sigma} \in \Upsilon_d \} \to \mathcal{T}_d (\Delta)_{\bm \sigma}$ defined by $\Upsilon \mapsto \Upsilon  \setminus {\sigma} \cup {\bm \sigma}$ is a bijection; the weight of $\Upsilon  \setminus {\sigma} \cup {\bm \sigma}$ is the weight of $\Upsilon$ divided by $x_{\sigma}$. Then there is a deletion/contraction-like relation:
\begin{align*}
\hat{k}_d(\Delta) &=
\sum_{\Upsilon\in\mathcal{T}_d(\Delta)\colon\sigma \notin \Upsilon_d}x_{\Upsilon}\,{\mathbf{t}_{d-1}(\Upsilon)}^2
+\sum_{\Upsilon\in\mathcal{T}_d(\Delta)\colon\sigma \in \Upsilon_d}x_{\Upsilon}\,{\mathbf{t}_{d-1}(\Upsilon)}^2
\\
&=
\sum_{\Upsilon\in\mathcal{T}_d(\Psi)}x_{\Upsilon}\,{\mathbf{t}_{d-1}(\Upsilon)}^2+x_{\sigma}\sum_{\Upsilon\in\mathcal{T}_d(\Delta^{\bm \sigma})\colon{\bm \sigma} \in \Upsilon_d}x_{\Upsilon}\,{\mathbf{t}_{d-1}(\Upsilon)}^2
\\
&=
\sum_{\Upsilon\in\mathcal{T}_d(\Psi)}x_{\Upsilon}\,{\mathbf{t}_{d-1}(\Upsilon)}^2+x_{\sigma}\sum_{\Upsilon\in\mathcal{T}_d(\Delta)_{\bm \sigma}}x_{\Upsilon}\,{\mathbf{t}_{d-1}(\Upsilon)}^2
\\
&=
\hat{k}_d{(\Psi)}+x_{\sigma}\,\hat{k}_d(\Delta)_{\bm \sigma}.
\end{align*}
Now Theorem~\ref{thm:fortree} and Equation~\eqref{def:resef} imply that
\[
\dfrac{\hat{k}_d(\Psi)}{\hat{k}_d(\Delta)}=\dfrac{\hat{k}_d(\Delta)-x_{\sigma}\,\hat{k}_d(\Delta)_{\bm \sigma}}{\hat{k}_d(\Delta)}
=1-x_\sigma R_\sigma
=1+\dfrac{i_\sigma}{i_{\bm\sigma}}
\]
where $R_\sigma$ is the simplicial effective resistance of $\sigma$ in $\Delta$.  Equivalently,
\begin{equation}\label{eq:Ratio}
\dfrac{\hat{k}_d(\Delta)}{\hat{k}_d(\Psi)}=\dfrac{i_{{\bm\sigma}}}{i_{{\bm\sigma}}+i_{{\sigma}}}.
\end{equation}

\begin{ex}
Let $\Psi,\Delta,\sigma$ be as in Example~\ref{ex:octa_ratio}.  Recall from Example~\ref{ex:octa_elec} that the simplicial effective resistance $R_\sigma$ in $\Delta$ is given by 
\[
R_{\sigma}=\dfrac{1}{x_{222}}\Big(\dfrac{x_{222}^*-x_{111}}{x_{222}^*}\Big).
\]
Therefore,
\[
\dfrac{\hat{k}_2(\Delta)}{\hat{k}_2(\Psi)}= \dfrac{1}{1-x_{\sigma}R_{\sigma}}=\dfrac{x_{222}^*}{x_{111}}=\dfrac{(x_{1,1}+x_{1,2})(x_{2,1}+x_{2,2})(x_{3,1}+x_{3,2})}{x_{1,1}x_{2,1}x_{3,1}},
\]
which matches the computation in Example~\ref{ex:octa_ratio}.
\end{ex}

\section{Color-shifted complexes}\label{sec:color-shifted}

In this section, we give a formula for weighted tree-numbers of \textit{color-shifted complexes} using the theory of high-dimensional electrical networks.  
The formula for the 1-dimensional case (i.e., for Ferrers graphs) appears in~\cite{EW}.

\begin{defn} \label{defn:CSC}
Fix a positive integer $d$.  For each color $q \in[d+1]$, let $V_q=\{v_{q,j} \st j=1,\dots,n_q\}$ be a set of vertices of color $q$.
A \emph{color-shifted complex} $\Delta$ on $V=V_1\cup\cdots\cup V_{d+1}$ is a simplicial complex on $V$ satisfying the following two conditions:
\begin{enumerate}
\item[(1)] each facet $\sigma \in \Delta$ has exactly one vertex of each color, i.e., $|\sigma \cap V_q|=1$ for each $q$ (in particular, $\Delta$ is pure of dimension~$d$), and
\item[(2)] if $\sigma \in \Delta$ and $v_{q,j} \in \sigma$, then $\sigma \setminus v_{q,j} \cup v_{q,j'} \in \Delta$ for every $j'<j$.
\end{enumerate}
\end{defn}

In light of condition~(1), henceforth, we represent a $d$-face $\tau=\{v_{1,j_1},\dots,v_{d+1,j_{d+1}}\}$ of $\Delta$ by the shorthand $(j_1,\dots,j_{d+1})$.  If we define \textit{componentwise order} on $d$-faces by $(j_1,\dots,j_{d+1})\leq(j'_1,\dots,j'_{d+1})$ iff $j_q\leq j'_q$ for every $q$, then
condition~(2) says that the facets of a color-shifted complex form an order ideal in componentwise order.
A color-shifted complex $\Delta$ is said to be \emph{generated} by a set $S$ of simplices if $\Delta$ is the smallest color-shifted complex containing $S$; in this case we write $\Delta=\shiftgen{S}$.
The unique minimal set of generators is the set of maximal elements of $\Delta$ with respect to componentwise order.
Color-shifted complexes with $d=1$ are precisely the \emph{Ferrers graphs} studied by Ehrenborg and van~Willigenburg \cite{EW}. 

We now describe the $d$th reduced homology group $\tilde{H}_d(\Delta)$ of a color-shifted complex $\Delta$ of dimension $d$.
Let $\Lambda$ be the subcomplex of $\Delta$ whose facets are
\begin{equation} \label{Lambda}
\Lambda_d=\{\tau=(j_1,\dots,j_{d+1}) \in \Delta_d\st j_q=1 \text{ for some } q\in[d+1]\}.
\end{equation}
For each $d$-face $\tau=(j_1,\dots,j_{d+1})\in \Delta_d\sm\Lambda_d$, let $C_\tau$ be the subcomplex of $\Delta$ with $2^{d+1}$ facets, each containing one of the two vertices $v_{q,1},v_{q,j_q}$ for each $q\in[d+1]$.  Observe that $C_\tau$ can be geometrically realized as the boundary of a $(d+1)$-crosspolytope, hence is a simplicial sphere.  There is a corresponding $d$-cycle
\begin{equation} \label{ztau}
z_\tau = \sum_{A\subseteq[d+1]} (-1)^{|A|} [\{v_{a,1}\st a\in A\}\cup\{v_{b,j_b}\st b\in[d+1]\sm A\}]
\end{equation}
in $\tilde{H}_d(\Delta)$.  The set $\{z_{\tau}\st \tau \in \Delta_d\sm\Lambda_d\}$
forms a basis of the $d$th reduced homology group $\tilde{H}_d(\Delta)$ and its rank is equal to $b_d(\Delta)=|\Delta_d|-|\Lambda_d|$ \cite[Thm.~5.7]{BN}.
Furthermore, $\Lambda$ is a spanning tree of $\Delta$.  To see this, first observe that the description~\eqref{Lambda} implies that $\Lambda^{(d-1)}=\Delta^{(d-1)}$.  Therefore, in the long exact sequence in relative homology for the pair $(\Delta,\Lambda)$
\[\cdots\to
H_{d-1}(\Delta,\Lambda;\Qq)\to H_d(\Lambda;\Qq)
\to H_d(\Delta;\Qq)
\xrightarrow{\delta} H_d(\Delta,\Lambda;\Qq)\to\cdots\]
the preceding discussion implies that $H_{d-1}(\Delta,\Lambda;\Qq)=0$ and that $\delta$ is an isomorphism, implying $H_d(\Lambda;\Qq)=0$.  Thus $\Lambda$ satisfies the first and third conditions of~\eqref{conditions-for-tree}.

\subsection{Simplicial effective resistances and the ratio of tree-numbers}

For a color $q \in[d+1]$ and a vertex $j\in[n_q]$, let $x_{q,j}$ be the weight of $v_{q,j}$.  The weight of a $d$-face $\tau=(j_1,\dots,j_{d+1})$ is thus $x_{\tau}=x_{1,j_1}\cdots x_{d+1,j_{d+1}}$.  We will find the explicit currents and voltages on the $d$-faces and obtain an expression for the simplicial effective resistance $R_{\sigma}$ in terms of the $x_{q,j}$. Our proof generalizes that of \cite[Prop.~2.2]{EW} for Ferrers graphs.
For $q \in[d+1]$, we abbreviate
\[
D_{q,j}=\sum_{i=1}^{j}x_{q,i}.
\]

\begin{thm}\label{thm:CSC-ratio}
Let $\Delta$ be a pure color-shifted complex of dimension $d$.  Let $\sigma=(\ell_1,\dots,\ell_{d+1})$ be a facet of $\Delta$ that is maximal in componentwise order, and such that $\ell_q\geq2$ for all $q$.
Then
\[
\dfrac{\hat{k}_d(\Delta)}{\hat{k}_d(\Delta\sm\sigma)}=\prod_{q=1}^{d+1}
\dfrac{D_{q,\ell_q}}
{D_{q,\ell_q-1}}.
\]
\end{thm}

Note that $\Delta$ has such a facet $\sigma$ if and only if $\Delta$ is not a tree.

\begin{proof}
Regard $\Delta$ as a resistor network with resistances $r_\tau=1/x_\tau$.  We attach a current generator~$\bm\sigma$ with vertices $v_{1,\ell_1},v_{2,\ell_2},\dots, v_{d+1,\ell_{d+1}}$, and current
\[i_{\bm\sigma}=\prod_{q=1}^{d+1}D_{q,\ell_q}.\]
For each $q \in[d+1]$ and $j_q \in[\ell_q-1]$, we set
\[U_{q,j_q}=-x_{q,j_q} \qquad\text{and}\qquad U_{q,\ell_q}=D_{q,\ell_q-1}.\]
Now, define a current vector $I_{\bm\sigma}=(i_\tau)_{\tau\in \Delta^{\bm\sigma}_d}$ by
\begin{equation} \label{color-shifted-current}
i_\tau = \begin{cases}
-i_{\bm\sigma} + \prod_{q=1}^{d+1}U_{q,\ell_q} & \text{ if } \tau=\sigma,\\
\rule[-4mm]{0mm}{10mm}
\prod_{q=1}^{d+1}U_{q,j_q} &\text{ if } \tau\in\shiftgen{\sigma}\setminus\{\sigma\},\\
0 &\text{ if } \tau\notin\shiftgen{\sigma}.
\end{cases}
\end{equation}
We regard the weights $x_\tau$ as conductances, so by OL~\eqref{OL} the voltage $V_{\bm\sigma}=(v_\tau)_{\tau\in \Delta^{\bm\sigma}_d}$ is given by
\begin{align*}
v_{\tau}&=i_\tau/x_\tau \text{  for } \tau \in \Delta_d,\\
v_{\bm\sigma} &= v_\sigma = i_\sigma/x_\sigma.
\end{align*}
We will show that the vectors $I_{\bm\sigma}$ and $V_{\bm\sigma}$ solve the simplicial network problem on $\Delta^{\bm\sigma}$, that is, that they respectively satisfy KCL and KVL.

\underline{Claim 1: $I_{\bm \sigma}$ satisfies KCL~\eqref{KCL}.}

 For $\eta \in \Delta_{d-1}$, let $(\partial_{\Delta^{\bm \sigma},d})_{\eta}$ be the row vector of $\partial_{\Delta^{\bm \sigma},d}$ corresponding to $\eta$. Since $\Delta_{d-1}$ is color-shifted, we may write $\eta=(j_1,\dots,\widehat{j_{s}},\dots,j_{d+1})$ for some $s\in[d+1]$, where the hat denotes removal.
 We will show that the quantity
\[(\partial_{\Delta^{\bm \sigma},d})_{\eta}\,I_{\bm \sigma}=
(-1)^{s-1}\left(\sum_{\tau \in \Delta^{\bm \sigma}_d\colon \eta \subseteq \tau}i_{\tau}\right)\]
is zero.  Let $\Psi=\Delta\sm\sigma$.  If $\eta \not\subseteq \sigma$, then
\[
\sum_{\tau \in \Delta^{\bm \sigma}_d\colon \eta \subseteq \tau}i_{\tau}
= \sum_{\tau \in \Psi_d\colon \eta \subseteq \tau}i_{\tau}
= \left(\prod_{q \in[d+1]\setminus s}U_{q,\,j_q}\right)
    \left(-\sum_{j_{s}=1}^{\ell_{s}-1}x_{s,\,j_{s}}+U_{s,\,\ell_{s}}\right)
= 0,
\]
while if $\eta \subseteq \sigma$, then
\begin{align*}
\sum_{\tau \in \Delta^{\bm \sigma}_d\colon \eta \subseteq \tau}i_{\tau}
&
=\sum_{\tau \in \Psi_d\colon \eta \subseteq \tau}i_{\tau}+i_{\sigma} + i_{\bm\sigma}
=\left(\sum_{\tau \in \Psi_d\colon \eta \subseteq \tau}i_{\tau}\right)+\prod_{q=1}^{d+1}U_{q,\,\ell_q} \\
&=
\left(\prod_{q \in[d+1]\setminus s}U_{q,\,\ell_q}\right)
\left(-\sum_{j_{s}=1}^{\ell_{s}-1}x_{s,\,j_{s}} \right)
+\prod_{q=1}^{d+1}U_{i,\,\ell_q}\\
&= 0,
\end{align*}
proving Claim~1.
\medskip

\underline{Claim 2: $V_{\bm \sigma}$ satisfies KVL~\eqref{KVL}.}

Recall that $\{z_\tau\st\tau\in\Delta_d\sm\Lambda_d\}$ (see~\eqref{ztau}) is a basis for $\tilde{H}_d(\Delta)=(\ker\partial_{\Delta,d})^\perp$,
so $\{z_\tau\st\tau\in\Delta_d\sm\Lambda_d\}\cup\{[\bm\sigma]-[\sigma]\}$ is a basis for $(\ker\partial_{\Delta^{\bm\sigma},d})^\perp$.
To prove Claim~2, it suffices to show that $V_{\bm\sigma}$ is orthogonal to each member of this basis.

First, evidently $V_{\bm\sigma}$ is orthogonal to $[\bm\sigma]-[\sigma]$.

Second, for $\tau\in\Delta_d\sm\Lambda_d$ with $\tau \ne \sigma$, write $\tau=(j_1,\dots,j_{d+1})$, using the notational shorthand introduced after Definition~\ref{defn:CSC}.  Note that $1<j_q<\ell_q$ for some $q \in[d+1]$.  For $A\subseteq[d+1]\sm\{q\}$, define $\eta_A=\{v_{a,1}\colon a\in A\}\cup\{v_{b,j_b}\colon b\in[d+1]\sm\{q\}\sm A\}$.  Then
\begin{align*}
\left\langle V_{\bm\sigma},z_\tau\right\rangle
&= \sum_{A\subseteq[d+1]\sm\{q\}} \left\langle V_{\bm\sigma},  [\eta_A\cup\{v_{q,j_q}\}]-[\eta_A\cup\{v_{q,1}\}] \right\rangle\\
&= \sum_{A\subseteq[d+1]\sm\{q\}} \left(\prod_{a\in A}U_{a,j_a} \prod_{b\in[d+1]\sm\{q\}\sm A} U_{b,j_b}\right)\left( \frac{U_{q,j_q}}{x_{\eta_A\cup\{v_{q,j_q}\}}} - \frac{U_{q,1}}{x_{\eta_A\cup\{v_{q,1}\}}} \right) \\
&= \sum_{A\subseteq[d+1]\sm\{q\}} \left(\prod_{a\in A}U_{a,j_a} \prod_{b\in[d+1]\sm\{q\}\sm A} U_{b,j_b}\right)\left( \frac{-x_{q,j_q}}{x_{\eta_A}x_{q,j_q}} + \frac{x_{q,1}}{x_{\eta_A}x_{q,1}} \right) && \text{(since $j_q<\ell_q$)}\\
&= 0.
\end{align*}

Finally,
\begin{align*}
\langle V_{\bm\sigma}, z_\sigma \rangle =
\langle V_{\bm\sigma}, z_{\bm\sigma} \rangle
&=\prod_{q=1}^{d+1}\left(-\dfrac{U_{q,\,1}}{x_{q,\,1}}+\dfrac{U_{q,\,\ell_q}}{x_{q,\,\ell_q}}\right)-\prod_{q=1}^{d+1}{\dfrac{D_{q,\,\ell_q}}{x_{q,\,\ell_q}}} \\
&= \prod_{q=1}^{d+1}\left(\dfrac{-x_{q,\,\ell_q}U_{q,\,1}+x_{q,\,1}U_{q,\,\ell_q}}{x_{q,\,1}x_{q,\,\ell_q}}\right)-\prod_{q=1}^{d+1}{\dfrac{D_{q,\,\ell_q}}{x_{q,\,\ell_q}}} \\
&= \prod_{q=1}^{d+1}\left(\dfrac{\sum_{j_q=1}^{\ell_q}{x_{q,\,j_q}}}{x_{q,\,\ell_q}}\right)-\prod_{q=1}^{d+1}{\dfrac{D_{q,\,\ell_q}}{x_{q,\,\ell_q}}}=0,
\end{align*}
proving Claim~2.
\medskip

To conclude the proof, we use~\eqref{eq:Ratio} to obtain
\[
\dfrac{\hat{k}_d(\Delta)}{\hat{k}_d(\Psi)}=\dfrac{i_{\bm\sigma}}{i_{\bm\sigma}+i_\sigma}=\prod_{q=1}^{d+1}
\dfrac{D_{q,\ell_q}}{D_{q,\ell_q-1}}.\qedhere
\]
\end{proof}

To illustrate the proof of Theorem~\ref{thm:CSC-ratio}, we present some examples. For notational convenience, abbreviate $i_{(j_1,j_2, \dots ,j_{d+1})}$ by $i_{j_1 j_2 \dots j_{d+1}}$.

\begin{ex}
Let $\sigma=v_{1,2}v_{2,2}v_{3,2}$, so that $\Delta=\shiftgen{\sigma}$ is the octahedron of Example~\ref{eq:octa}.  Then
\begin{align*}
i_{\bm\sigma}&=(x_{1,1}+x_{1,2})(x_{2,1}+x_{2,2})(x_{3,1}+x_{3,2}) &&= D_{1,2}D_{2,2}D_{3,2},\\
i_{\sigma}&=x_{1,1}x_{2,1}x_{3,1}-i_{\bm\sigma} &&= D_{1,1}D_{2,1}D_{3,1}-i_{\bm\sigma}.
\end{align*}
Per \eqref{color-shifted-current}, the currents of the other facets are
\begin{align*}
i_{111} &=-x_{1,1}x_{2,1}x_{3,1}, &	i_{121} &= x_{1,1}D_{2,1}x_{3,1}, &	i_{211}&=D_{1,1}x_{2,1}x_{3,1}, &	i_{221} &=-D_{1,1}D_{2,1}x_{3,1}\\
 i_{112} &=x_{1,1}x_{2,1}D_{3,1}, &	i_{122} &=-x_{1,1}D_{2,1}D_{3,1}, &	i_{212} &=-D_{1,1}x_{2,1}D_{3,1},
\end{align*}
which coincide with the currents given in Example~\ref{ex:octa_elec}.  Therefore
\[
\dfrac{\hat{k}_2{(\Delta)}}{\hat{k}_2(\Delta\setminus\sigma)}= \dfrac{D_{1,2}D_{2,2}D_{3,2}}{D_{1,1}D_{2,1}D_{3,1}}=\dfrac{(x_{1,1}+x_{1,2})(x_{2,1}+x_{2,2})(x_{3,1}+x_{3,2})}{x_{1,1}x_{2,1}x_{3,1}}.
\]
\end{ex}

\begin{ex}
Let $\Delta$ be the $2$-dimensional color-shifted complex $\shiftgen{\sigma}$, where $\sigma=v_{1,2}v_{2,2}v_{3,3}$.  Then
\begin{align*}
i_{\bm\sigma}&=(x_{1,1}+x_{1,2})(x_{2,1}+x_{2,2})(x_{3,1}+x_{3,2}+x_{3,3}) &&= D_{1,2}D_{2,2}D_{3,3},\\
i_{\sigma}&=x_{1,1}x_{2,1}(x_{3,1}+x_{3,2})-i_{\bm\sigma} &&= D_{1,1}D_{2,1}D_{3,2}-i_{\bm\sigma}.
\end{align*}
Per \eqref{color-shifted-current}, the currents of the other facets are
\begin{align*}
i_{111} &= -x_{1,1}x_{2,1}x_{3,1}, &	i_{121} &= x_{1,1}D_{2,1}x_{3,1}, &	i_{211} &= D_{1,1}x_{2,1}x_{3,1}, &	i_{221} &= -D_{1,1}D_{2,1}x_{3,1},\\
i_{112} &= -x_{1,1}x_{2,1}x_{3,2}, &	i_{122} &= x_{1,1}D_{2,1}x_{3,2}, &	i_{212} &= D_{1,1}x_{2,1}x_{3,2}, &	i_{222} &= -D_{1,1}D_{2,1}x_{3,2},\\
i_{113} &= x_{1,1}x_{2,1}D_{3,2}, &	i_{123} &= -x_{1,1}D_{2,1}D_{3,2}, &	i_{213} &= -D_{1,1}x_{2,1}D_{3,2}.
\end{align*}
Thus
\[
\dfrac{\hat{k}_2{(\Delta)}}{\hat{k}_2(\Delta\setminus\sigma)}= \dfrac{D_{1,2}D_{2,2}D_{3,3}}{D_{1,1}D_{2,1}D_{3,2}}=\dfrac{(x_{1,1}+x_{1,2})(x_{2,1}+x_{2,2})(x_{3,1}+x_{3,2}+x_{3,3})}{x_{1,1}x_{2,1}(x_{3,1}+x_{3,2})}.
\]
\end{ex}

\begin{ex}
Let $G$ be the Ferrers graph generated by $\sigma=v_{1,\ell_1}v_{2,\ell_2}$. Then
\[
i_{\bm\sigma}=D_{1,\ell_1}D_{2,\ell_2} \quad\text{and}\quad i_{\sigma}=D_{1,\ell_1-1}D_{2,\ell_2-1}-i_{\bm\sigma}.
\]
For $j_1 \in[\ell_1-1]$ and $j_2 \in[\ell_2-1]$, 
\[
i_{j_1 \ell_2}=-x_{1,j_1}D_{2,\ell_2-1}, \quad i_{\ell_1 j_2}=-D_{1,\ell_1-1}x_{2,j_2}, \quad i_{j_1 j_2}=x_{j_1}x_{j_2}, 
\]
and the other currents are zero. Then we have
\[
\dfrac{\hat{k}_1{(G)}}{\hat{k}_1(G\setminus\sigma)}= \dfrac{D_{1,\ell_1}D_{2,\ell_2}}{D_{1,\ell_1-1}D_{2,\ell_2-1}}
\]
which recovers~\cite[Prop.~2.2]{EW}.
\end{ex}

\subsection{Tree enumeration in color-shifted complexes}

We now use simplicial effective resistance to enumerate simplicial spanning trees of color-shifted complexes.  The main result, Theorem~\ref{thm:CSC-enumeration}, was first conjectured by G.~Aalipour and the first author [unpublished].

Let $\Delta$ be a $d$-dimensional color-shifted complex (so there are $d+1$ colors).  Note that every $(d-1)$-face $\rho \in \Delta$ is missing a unique color $m(\rho)$; set $k(\rho) = \max\{j\st \rho \cup \{v_{m(\rho),j}\} \in \Delta\}$.

Let $\Lambda$ be the spanning tree of $\Delta$ with facets $\Lambda_d$ given by~\eqref{Lambda}.
Let $\Gamma$ be the subcomplex of $\Delta$ whose facets are the facets {\em not} in $\Lambda$, i.e.,
\[\Gamma = \{\rho \in \Delta\st v_{q,1} \not\in \rho\ \text{for all}\ q\}.\]

\begin{thm}\label{thm:CSC-enumeration}
\[
\hat{k}_d(\Delta) = \prod_{q,i} x_{q,i}^{e(q,i)} \prod_{\rho \in \Gamma_{d-1}}D_{m(\rho),k(\rho)}
\]
where 
\[
e(q,i) = \#\{\sigma \in \Delta_d\st v_{q,i} \in \sigma\ \text{and}\ v_{r,1} \in \sigma\ \text{for some}\ r \neq q\}.
\]
\end{thm}

\begin{proof}
We construct $\Delta$ recursively, starting with the spanning tree $\Lambda$. Since it is a tree, 
$\hat{k}_d(\Lambda)$ is simply the product of the weights of its facets.  Let
\[
f(q,i) = \#\{\sigma \in \Lambda_d\st v_{q,i} \in \sigma\}.
\]

Then, if $i>1$,
\[
f(q,i) = \#\{\sigma \in \Delta_d\st 
	 v_{q,i} \in\sigma\ \text{and}\ v_{r,1} \in \sigma\ \text{for some}\ r \neq q \}
	= e(q,i),
\]
but for $i=1$,
\begin{align*}
	f(q,1) &= \#\{\sigma \in \Delta_d\st 
		v_{q,1} \in \sigma\ \text{and}\ v_{r,1} \in \sigma\ \text{for some}\ r \neq q \} \\
		&\quad + \#\{\sigma \in \Delta_d\st 
			v_{q,1} \in \sigma\ \text{and}\ v_{r,1} \not\in \sigma\ \text{for any}\ r \neq q \}\\
		&= e(q,1) + \#\{\rho \in \Gamma_{d-1}\st m(\rho) = q \}.
\end{align*}
Therefore, 
\begin{equation}\label{kdLambda}
	\hat{k}_d(\Lambda) 
        = \prod_{\sigma \in \Lambda_d}\prod_{v_{q,i} \in \sigma} x_{q,i}
		= \prod_{q, i} x_{q,i}^{f(q,i)}
		= \prod_{q, i} x_{q,i}^{e(q,i)}\prod_{\rho \in \Gamma_{d-1}} x_{m(\rho),1}.
\end{equation}

Next we turn our attention to the remainder of $\Delta$. 
For each facet $\sigma \in \Gamma$, define $\Delta_{< \sigma}$ to be the complex generated by all of $\Lambda$ and the facets that were already added in before $\sigma$ (the facets lexicographically earlier than $\sigma$).  
Since $\sigma \in \Gamma$, we know $j>1$ for each $v_{q,j} \in \sigma$, so
we can apply Theorem~\ref{thm:CSC-ratio} to compute
\begin{equation}\label{Delta}
	\hat{k}_d(\Delta)
	= \hat{k}_d(\Lambda)\prod_{\sigma \in \Gamma_d}
		\frac{\hat{k}_d(\Delta_{< \sigma} \cup \sigma)}{\hat{k}_d(\Delta_{< \sigma})} 
	= \hat{k}_d(\Lambda)\prod_{\sigma \in \Gamma_d}
		\prod_{v_{q,j} \in \sigma} \frac{D_{q,j}}{D_{q,j-1}}
	= \hat{k}_d(\Lambda)\prod_{q}\prod_{\sigma \in \Gamma_d}\prod_{j\st v_{q,j} \in \sigma}
		 \frac{D_{q,j}}{D_{q,j-1}}.
\end{equation}

Now, for fixed $q$, let $\Gamma_{d-1,q}$ denote the set of ridges (($d-1$)-dimensional faces) of $\Gamma$ that are missing color $q$ (i.e., faces $\rho \in \Gamma_{d-1}$ such that $m(\rho) = q$).
We can then write any $\sigma \in \Gamma_d$ as $\sigma = (\rho,j)$, where $\rho$ is the unique ridge of $\sigma$ in $\Gamma_{d-1,q}$ and $v_{q,j}$ is the unique vertex of color $q$ in $\sigma$.
Then, for fixed $q$,
\begin{equation}\label{D-ratio}
	\prod_{\sigma \in \Gamma_d}\prod_{j\st v_{q,j} \in \sigma}
		 \frac{D_{q,j}}{D_{q,j-1}}
	= \prod_{\rho \in \Gamma_{d-1, q}}\prod_{(\rho, j)\in\Gamma} 
		\frac{D_{q,j}}{D_{q,j-1}}
	= \prod_{\rho \in \Gamma_{d-1, q}}\prod_{j=2}^{k(\rho)}
		\frac{D_{q,j}}{D_{q,j-1}}
	= \prod_{\rho \in \Gamma_{d-1, q}}\frac{D_{q,k(\rho)}}{D_{q,1}}.
\end{equation}
Combining Equations~\eqref{kdLambda},~\eqref{Delta}, and~\eqref{D-ratio},
\begin{align*}
	\hat{k}_d(\Delta)
	&= \hat{k}_d(\Lambda)\prod_{q}\prod_{\sigma \in \Gamma_d}\prod_{j\st v_{q,j} \in \sigma}
		\frac{D_{q,j}}{D_{q,j-1}}
	= \hat{k}_d(\Lambda)\prod_{q}\prod_{\rho \in \Gamma_{d-1, q}}\frac{D_{q,k(\rho)}}{D_{q,1}}\\
	&= \prod_{q, i} x_{q,i}^{e(q,i)}\prod_{\rho \in \Gamma_{d-1}} x_{m(\rho),1}
		\prod_{\rho \in \Gamma_{d-1}}\frac{D_{m(\rho),k(\rho)}}{D_{m(\rho),1}}
	= \prod_{q, i} x_{q,i}^{e(q,i)}
		\prod_{\rho \in \Gamma_{d-1}}D_{m(\rho),k(\rho)}\\
\end{align*}
since $D_{q,1} = x_{q,1}$, for any color $q$.
\end{proof}

\begin{ex} \label{ex:ccc}
Let $\Delta$ be the complete colorful complex on $V=V_1\cup \cdots \cup  V_{d+1}$, i.e., the simplicial join $V_1*\cdots*V_{d+1}$.  Then $\Delta$ is certainly color-shifted, and $\Gamma=V'_1*\cdots*V'_{d+1}$, where $V'_q=V_q\sm\{v_{q,1}\}$.  Moreover, setting $n_q=\#V_q$ for each $q\in[d+1]$,
\[e(q,i) = \prod_{r\neq q} n_r - \prod_{r\neq q}(n_r-1) = \sum_{\substack{J\subseteq[d+1]\sm\{q\}\\ |J|\leq d-1}} (-1)^{d-1-|J|} \prod_{r\in J} n_{r}\]
(a quantity independent of $i$, and denoted by $E_{d,q}$ in \cite{ADKLM}).
Therefore, the formula of Theorem~\ref{thm:CSC-enumeration} becomes
\begin{align*}
\hat{k}_d(\Delta)
&= \prod_{q,i} x_{q,i}^{e(q,i)} \prod_{q=1}^{d+1} \prod_{\rho\in V'_1\times\cdots\times\widehat{V'_q}\times\cdots\times V'_{d+1}} D_{q,n_q}\\
&= \prod_{q=1}^{d+1} \left[ \left(\prod_{i=1}^{n_i} x_{q,i}\right)^{e(q,i)}
(x_{q,1}+\cdots+x_{q,n_q})^{\prod_{r\neq q}(n_{r}-1)} \right]
\end{align*}
which is a special case of \cite[Thm.~1.2]{ADKLM} (setting $r=d+1$ and $k=d$ therein), and further specializing $x_{q,i}=1$ for all $q,i$ recovers the unweighted tree enumerator
\[ \prod_{q=1}^{d+1} n_q^{\prod_{r\neq q}(n_{r}-1)} \]
which is the $k=r-1$ case of Adin's formula~\cite[Thm.~1.5]{A}.
\end{ex}

\begin{ex}
Consider the color-shifted complex $\Delta=\shiftgen{\Green{2}\Red{3}\Blue{5}, \Green{3}\Red{2}\Blue{4}, \Green{3}\Red{3}\Blue{3}}$.  Since there are three colors, we can represent it as a plane partition (the 3-dimensional analogue of a Ferrers diagram), as shown in Figure~\ref{fig:Delta}, 
where the first (green) coordinate points to the left, the second (red) coordinate points to the right, and the third (blue) coordinate points up.  
Each block represents a facet of $\Delta$; the outside corners are the color-shifted generators.  The subcomplexes $\Lambda$ and $\Gamma$ correspond to the ``boundary'' and ``interior'' of the plane partition; see Figures~\ref{fig:Lambda} and~\ref{fig:Gamma}.
 
\begin{figure}[ht]
     \centering
     \begin{subfigure}[b]{0.3\textwidth}
         \centering
         	\begin{tikzpicture}[scale=.8]
			\planepartition{{5,5,5},{5,5,5},{4,4,3}}
			\foreach \pct/\xa/\ya/\xb/\yb in {.55/2/5/1/3, 0/-2/-3/-1.2/1, .7/2/-3/.3/-.3}
				{
				\draw[black](\xa,\ya) -- ({\xa*\pct+\xb*(1-\pct)}, {\ya*\pct+\yb*(1-\pct)});
				\draw[white] ({\xa*\pct+\xb*(1-\pct)}, {\ya*\pct+\yb*(1-\pct)}) -- (\xb,\yb);
				}
			\node[fill=white] at (2,5) {\Green{2}\Red{3}\Blue{5}};
			\node[fill=white] at (-2,-3) {\Green{3}\Red{2}\Blue{4}};
			\node[fill=white] at (2,-3) {\Green{3}\Red{3}\Blue{3}};
			\end{tikzpicture}
          \caption{$\Delta$}
         \label{fig:Delta}
     \end{subfigure}
     \hfill
     \begin{subfigure}[b]{0.3\textwidth}
         \centering
         	\begin{tikzpicture}[scale=.8]
			\planepartition{{5,5,5},{5,1,1},{4,1,1}}
			\end{tikzpicture}
		 \caption{$\Lambda$}
         \label{fig:Lambda}
     \end{subfigure}
     \hfill
     \begin{subfigure}[b]{0.3\textwidth}
         \centering
         	\begin{tikzpicture}[scale=.8]
			\planepartition{{4,4},{3,2}}
			\end{tikzpicture}\\ [5ex]
         \caption{$\Gamma$}
         \label{fig:Gamma}
     \end{subfigure}
        \caption{The color-shifted complex $\Delta$ generated by $\Green{2}\Red{3}\Blue{5}, \Green{3}\Red{2}\Blue{4}, \Green{3}\Red{3}\Blue{3}$.}
        \label{fig:threeD}
\end{figure}
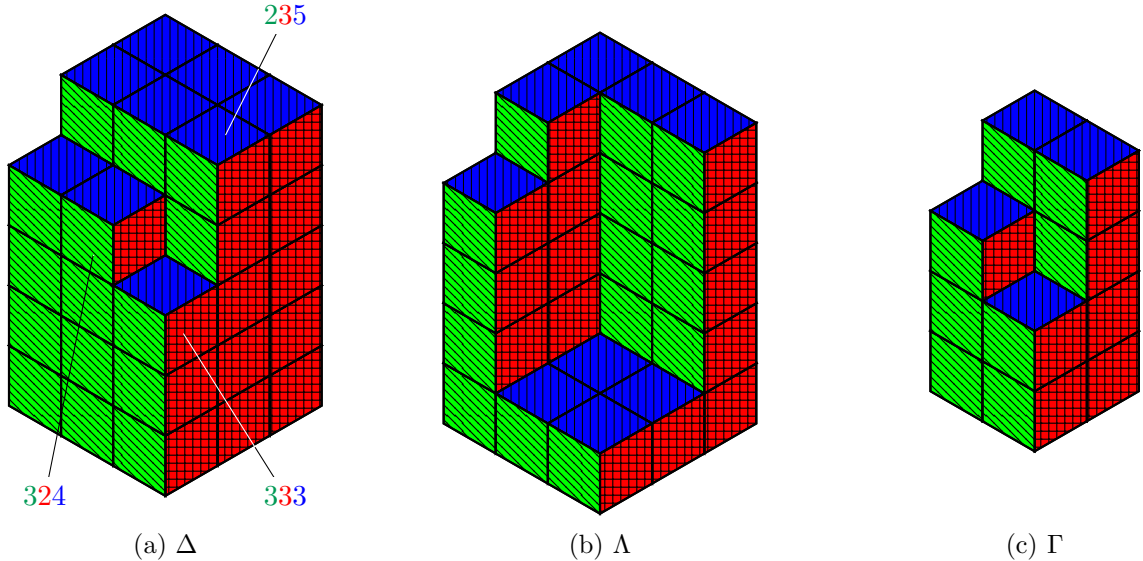

The contribution from $\Lambda$ depends on computing $f$ and $e$.  For instance, consider $f(\Blue{3},1),\ldots, f(\Blue{3},5)$ and $e(\Blue{3},1), \ldots, e(\Blue{3},5)$.  The values for $f(\Blue{3},i)$ are just given by the number of facets in the diagram of $\Lambda$ whose  height (coordinate in the blue direction) is $i$. So $f(\Blue{3},1) = 9$ because there are 9 facets in $\Lambda$ whose height is 1.  But 4 of those facets (the ones that are visible in the diagram of $\Lambda$) do not contribute to $e$, because in those facets, blue is the only color whose coordinate is 1, so $e(\Blue{3},1) = 9 - 4 = 5$.  (Those 4 facets will be exactly the ones that are cancelled out by the contribution from $\Gamma$.)  Next, $f(\Blue{3}, i) = e(\Blue{3}, i) = 5$ for $i = 2, 3, 4$, because in each case there are 5 facets whose height is $i$, but $f(\Blue{3}, 5) = e(\Blue{3}, 5) = 4$, because there are only 4 facets in $\Lambda$ whose height is 5.  

Similarly, measuring distance in the green direction, $f(\Green{1}, 1) = 15$ and $e(\Green{1}, 1) = f(\Green{1}, 2) = e(\Green{1}, 2) = 7,\ f(\Green{1}, 3) = e(\Green{1}, 3) = 6$.
Measuring distance in the red direction, $f(\Red{2}, 1) = 14$ and $e(\Red{2}, 1) = f(\Red{2}, 2) = e(\Red{2}, 2) = f(\Red{2}, 3) = e(\Red{2}, 3) = 7$.  
Therefore
\begin{align*}
\hat{k}_d(\Lambda) 
	&= 	(x_{\Green{1},1}^{15} x_{\Green{1},2}^{7} x_{\Green{1},3}^{6})
		(x_{\Red{2},1}^{14} x_{\Red{2},2}^{7} x_{\Red{2},3}^{7})
		(x_{\Blue{3},1}^9 x_{\Blue{3},2}^5 x_{\Blue{3},3}^5 x_{\Blue{3},4}^5 x_{\Blue{3},5}^4)\\
	&= 	(x_{\Green{1},1}^{7} x_{\Green{1},2}^{7} x_{\Green{1},3}^{6})
		(x_{\Red{2},1}^{7} x_{\Red{2},2}^{7} x_{\Red{2},3}^{7})
		(x_{\Blue{3},1}^5 x_{\Blue{3},2}^5 x_{\Blue{3},3}^5 x_{\Blue{3},4}^5 x_{\Blue{3},5}^4)
		(x_{\Green{1},1}^8 x_{\Red{2},1}^7 x_{\Blue{3},1}^4),
\end{align*}
matching Equation~\eqref{kdLambda}.

Next we consider the contribution from $\Gamma$.  The contribution from the face $\Green{2}\Red{3}\Blue{5}$ (so the green coordinate is 2, the red coordinate is 3, the blue coordinate is 5), for instance, is 
\[
	\frac{D_{\Green{1},2}}{D_{\Green{1},1}} \frac{D_{\Red{2},3}}{D_{\Red{2},2}} \frac{D_{\Blue{3},5}}{D_{\Blue{3},4}}
\]
To illustrate Equation~\eqref{D-ratio}, let us first focus just on blue.  There are 4 ridges $\rho$ in $\Gamma$ whose missing color $m(\rho)$ is blue, corresponding to the 4 visible blue faces at the bottom of the diagram of Figure~\ref{fig:Lambda}.  If we sort the blue contributions from all facets in $\Gamma$ by which of these 4 ridges they ``sit on top of'', we get
\begin{align*}
	\frac{D_{\Blue{3},5}}{D_{\Blue{3},4}} \frac{D_{\Blue{3},4}}{D_{\Blue{3},3}} \frac{D_{\Blue{3},3}}{D_{\Blue{3},2}} \frac{D_{\Blue{3},2}}{D_{\Blue{3},1}} &= \frac{D_{\Blue{3},5}}{D_{\Blue{3},1}}; \\
	\frac{D_{\Blue{3},5}}{D_{\Blue{3},4}} \frac{D_{\Blue{3},4}}{D_{\Blue{3},3}} \frac{D_{\Blue{3},3}}{D_{\Blue{3},2}} \frac{D_{\Blue{3},2}}{D_{\Blue{3},1}} &= \frac{D_{\Blue{3},5}}{D_{\Blue{3},1}}; \\
	\frac{D_{\Blue{3},4}}{D_{\Blue{3},3}} \frac{D_{\Blue{3},3}}{D_{\Blue{3},2}} \frac{D_{\Blue{3},2}}{D_{\Blue{3},1}} &= \frac{D_{\Blue{3},4}}{D_{\Blue{3},1}}; \\
	\frac{D_{\Blue{3},3}}{D_{\Blue{3},2}} \frac{D_{\Blue{3},2}}{D_{\Blue{3},1}} &= \frac{D_{\Blue{3},3}}{D_{\Blue{3},1}}
\end{align*}
for 
ridges $\Green{2}\Red{2}, \Green{2}\Red{3}, \Green{3}\Red{2}, \Green{3}\Red{3}$, 
respectively, so the total blue contribution is  
\[
	\frac{D_{\Blue{3},5}^2 D_{\Blue{3},4}D_{\Blue{3},3}}{D_{\Blue{3},1}^4}.
\]
This exemplifies Equation~\eqref{D-ratio}, since $k(\Green{2}\Red{2}) = k(\Green{2}\Red{3}) = 5$, $k(\Green{3}\Red{2}) = 4$, and $k(\Green{3}\Red{3}) = 3$, so there are two ridges in $\Gamma$ whose maximum blue coordinate is $5$, one ridge whose maximum blue coordinate is $4$, and one ridge whose maximum blue coordinate is $3$.  See Figure~\ref{fig:blueGamma}.  

\begin{figure}
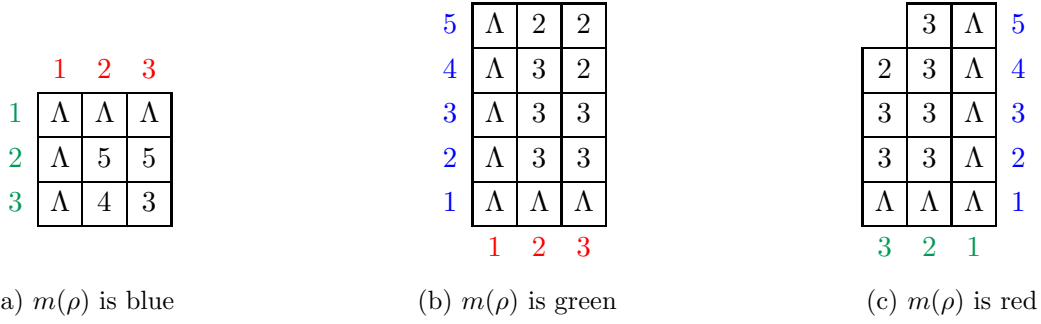

     \centering
     \begin{subfigure}[b]{0.3\textwidth}
         \centering
         	\begin{ytableau}
			\none    & \none[\Red{1}] & \none[\Red{2}] & \none[\Red{3}] \\
			\none[\Green{1}] & \Lambda & \Lambda  & \Lambda  \\
			\none[\Green{2}] & \Lambda & 5 & 5  \\
			\none[\Green{3}] & \Lambda & 4 & 3\\
			\none[ ] 
			\end{ytableau}
         \caption{$m(\rho)$ is blue}
         \label{fig:blueGamma}
     \end{subfigure}
     \hfill
     \begin{subfigure}[b]{0.3\textwidth}
         \centering
			\begin{ytableau}
			\none[\Blue{5}] & \Lambda & 2 & 2  \\
			\none[\Blue{4}] & \Lambda & 3 & 2  \\
			\none[\Blue{3}] & \Lambda & 3 & 3  \\
			\none[\Blue{2}] & \Lambda & 3 & 3  \\
			\none[\Blue{1}] & \Lambda & \Lambda & \Lambda \\
			\none & \none[\Red{1}] & \none[\Red{2}] & \none[\Red{3}] 
			\end{ytableau}
		 \caption{$m(\rho)$ is green}
         \label{fig:greenGamma}
     \end{subfigure}
     \hfill
     \begin{subfigure}[b]{0.3\textwidth}
         \centering
			\begin{ytableau}
			\none & 3 & \Lambda & \none[\Blue{5}]\\
			2 & 3  & \Lambda & \none[\Blue{4}] \\
			3 & 3  & \Lambda & \none[\Blue{3}] \\
			3 & 3  & \Lambda & \none[\Blue{2}] \\
			\Lambda & \Lambda & \Lambda & \none[\Blue{1}] \\
			\none[\Green{3}] & \none[\Green{2}] & \none[\Green{1}] & \none
			\end{ytableau}
         \caption{$m(\rho)$ is red}
         \label{fig:redGamma}
     \end{subfigure}
        \caption{Each ridge $\rho$ of $\Delta$, sorted by its missing color $m(\rho)$.  A ridge $\rho$ with incident facet(s) in $\Gamma$ is labeled with $k(\rho)$, the maximum value of that missing color that can be added to $\rho$.  A ridge whose incident facets are all in $\Lambda$ is labeled with $\Lambda$.}
        \label{fig:GammaRidges}
\end{figure}

Similarly, the green contribution is
\[
	\frac{D_{\Green{1},3}^5 D_{\Green{1},2}^3}{D_{\Green{1},1}^8}
\]
and the red contribution is 
\[
	\frac{D_{\Red{2},3}^6 D_{\Red{2},2}}{D_{\Red{2},1}^7}.
\]
See Figures~\ref{fig:greenGamma} and~\ref{fig:redGamma}, respectively.

Therefore,  
\begin{align*}
	\hat{k}_d(\Delta)
	&= x_{\Green{1},1}^{7} x_{\Green{1},2}^{7} x_{\Green{1},3}^{6}\cdot
		x_{\Red{2},1}^{7} x_{\Red{2},2}^{7} x_{\Red{2},3}^{7}\cdot
		x_{\Blue{3},1}^5 x_{\Blue{3},2}^5 x_{\Blue{3},3}^5 x_{\Blue{3},4}^5 x_{\Blue{3},5}^4\cdot
		x_{\Green{1},1}^8 x_{\Red{2},1}^7 x_{\Blue{3},1}^4\cdot
		\frac{D_{\Green{1},3}^5 D_{\Green{1},2}^3}{D_{\Green{1},1}^8}
		\frac{D_{\Red{2},3}^6 D_{\Red{2},2}}{D_{\Red{2},1}^7}
		\frac{D_{\Blue{3},5}^2 D_{\Blue{3},4} D_{\Blue{3},3}}{D_{\Blue{3},1}^4}\\
	&= x_{\Green{1},1}^{7} x_{\Green{1},2}^{7} x_{\Green{1},3}^{6}\cdot
		x_{\Red{2},1}^{7} x_{\Red{2},2}^{7} x_{\Red{2},3}^{7}\cdot
		x_{\Blue{3},1}^5 x_{\Blue{3},2}^5 x_{\Blue{3},3}^5 x_{\Blue{3},4}^5 x_{\Blue{3},5}^4\cdot
		D_{\Green{1},3}^5 D_{\Green{1},2}^3
		D_{\Red{2},3}^6 D_{\Red{2},2}
		D_{\Blue{3},5}^2 D_{\Blue{3},4} D_{\Blue{3},3}.
\end{align*}
\end{ex}

\section{Shifted complexes}\label{sec:shifted}
In this section we use high-dimensional electrical networks to give a new derivation of the weighted tree-numbers of shifted complexes, previously calculated in \cite{DKM1}.  The formula for the 1-dimensional case (i.e., for threshold graphs) appears in~\cite{MR,RW}. 

\begin{defn}
A (pure) \emph{shifted} complex $\Delta$ on a totally ordered vertex set $V$ (typically an interval) is a simplicial complex satisfying the following condition:
\begin{center}
If $\tau \in \Delta$ and $q \in \tau$, then $\tau \setminus q \cup r \in \Delta$ for every $r<q$.
\end{center}
Equivalently, define the \textbf{Gale order} on $k$-sets of positive integers by
$\{a_1<\cdots<a_k\} \leq \{b_1<\cdots<b_k\}$ if $a_i \leq b_i$ for all $1 \leq i \leq k$.
Then a complex is shifted if its facets form an order ideal in Gale order.
\end{defn}

Let $\Delta$ be a shifted complex of dimension $d$.  Define subcomplexes
\[
\Gamma = \{\tau \in \Delta\st 1 \not\in \tau\}, \qquad \Lambda = \{\rho \in \Gamma\st \rho \dju \{1\} \in \Delta\}.
\]
Note that both $\Gamma$ and $\Lambda$ are shifted.  In general $\dim\Gamma=d$ (unless $\Delta$ is a cone over vertex~1), while $\dim\Lambda=d-1$. 
The complex $1 \ast \Lambda = \{\sigma \in \Delta\st 1 \in \sigma\}$ is a spanning tree of $\Delta$.  
Moreover, each $d$-face $\tau=\{j_1,\dots,j_{d+1}\} \in \Gamma_d$ gives rise to a $d$-cycle
\begin{equation} \label{shifted-cycle}
z_{\tau}=\partial_d[1,j_1,\dots,j_{d+1}] = [\tau]+\sum_{s=1}^{d+1} (-1)^s[\tau\sm j_s\cup 1]
\in\tilde{H}_d(\Delta)
\end{equation}
and the set $\{z_\tau\st \tau \in \Gamma_d\}$ forms a basis of $\tilde{H}_d(\Delta)$; in particular $b_d(\Delta)=|\Gamma_d|$ \cite[Thm.~4.3]{BK}.  

\subsection{Simplicial effective resistances and the ratio of tree-numbers}
We shall find \emph{explicit} expressions for a current vector and a voltage vector, as in the proof of Theorem~\ref{thm:CSC-ratio}, and obtain the ratio of tree-numbers of shifted complexes. The proof of the following theorem requires more careful analysis concerning signs than that of Theorem~\ref{thm:CSC-ratio}. 

For $q \in[n]$, let $x_{q}$ be the weight of the vertex $q$, and define the weight of a $d$-face $\tau=j_1\dots j_{d+1}$ as $x_{\tau}=x_{j_1}\cdots x_{j_{d+1}}$. For $q \in[n]$, write
\[
D_{q}=\sum_{r=1}^{q}x_{r}.
\]
Also, a shifted complex $\Delta$ is said to be \emph{generated} by a set $S$ of simplices if $\Delta$ is the smallest shifted complex containing $S$.  In this case we write $\Delta=\shiftgen{S}$.  (In this section, the notation will always refer to shifted rather than color-shifted complexes.)

\begin{thm}\label{thm:SC-ratio}
Let $\Delta$ be a pure shifted complex of dimension $d$.  Let $\sigma=\{\ell_1<\cdots<\ell_{d+1}\}$ be a facet of $\Delta$ that is maximal in Gale order, and such that $\ell_1\geq2$.  Then
\[
\dfrac{\hat{k}_d(\Delta)}{\hat{k}_d{(\Delta\sm\sigma)}}=\prod_{q=1}^{d+1}\dfrac{D_{\ell_q}}{D_{\ell_q-1}}.
\]
\end{thm}
Note that $\Delta$ has such a facet $\sigma$ if and only if $\Delta$ is not a tree.

\begin{proof}
As in the proof of Theorem~\ref{thm:CSC-ratio}, we regard $\Delta$ as a resistor network with resistances $r_\tau=1/x_\tau$ and attach a current generator~$\bm\sigma$ with vertices $\ell_1,\dots,\ell_{d+1}$ and current
\[i_{\bm\sigma}=\prod_{q=1}^{d+1}D_{\ell_q}.\]
We define $\ell_0=0$.

Now, define a current vector $I_{\bm\sigma}=(i_\tau)_{\tau\in \Delta^{\bm\sigma}_d}$ as follows.  Set
\begin{align}
i_\tau&=0  \text{ for all } \tau\not\in\shiftgen{\sigma}, \label{itau:0}\\
i_\sigma&=-i_{\bm\sigma}+\prod_{q=1}^{d+1}D_{\ell_q-1} =\prod_{q=1}^{d+1}D_{\ell_q-1} - \prod_{q=1}^{d+1}D_{\ell_q}.\label{isigma}
\end{align}
Suppose that $\tau=\{j_1<\cdots<j_{d+1}\}\in \shiftgen{\sigma}\sm\{\sigma\}$, i.e., $j_q\leq\ell_q$ for all $q \in[d+1]$, with at least one equality strict.  We define the current on $\tau$ as follows.
First, let $\pi(\tau)$ be the unique permutation on $[d+1]$ such that (a) $\pi(\tau)(q)=r$ if $j_q=\ell_r$ for some $r \in[d+1]$, and (b) $\pi(\tau)$ is (b) $\pi(\tau)$ is increasing on $[d+1]\sm \{q\in[d+1]:j_q \in\sigma\cap\tau\}$.
Second, let
\begin{subnumcases}{U^\tau_q =}
0
	& if $j_q < \ell_{q-1}$, \label{Utauq:1}
\\
D_{j_{q}}=D_{\ell_{q-1}}
	& if $j_q=\ell_{q-1}$, \label{Utauq:2} 
\\
-x_{j_q}
	& if $\ell_{q-1}< j_q<\ell_q$ (i.e., $j_q \in \tau \setminus \sigma$), \label{Utauq:3}  
\\
D_{j_q-1}=D_{\ell_q-1}
	& if $j_q=\ell_q$. \label{Utauq:4}
\end{subnumcases}
Finally, define $i_{\tau}=(-1)^{\inv(\pi(\tau))}\prod_{q=1}^{d+1}U^\tau_q$, where as before $\inv(\pi(\tau))$ denotes the number of inversions.  In summary, the entries of the current vector $I_{\bm\sigma}$ are
\begin{subnumcases}{i_\tau =}
\textstyle -i_{\bm\sigma}+\prod_{q=1}^{d+1}D_{\ell_q-1}
	& if $\tau=\sigma$, \label{current:sigma}
\\
\textstyle (-1)^{\inv(\pi(\tau))}{\prod_{q=1}^{d+1}U^{\tau}_{q}}
	& if $\tau\in\shiftgen{\sigma}\sm\{\sigma\}$, \label{current:insigma}
\\
0
	& if $\tau\not\in\shiftgen{\sigma}$. \label{current:zero}
\end{subnumcases}
We regard the weights $x_\tau$ as conductances, so by OL~\eqref{OL} the voltage vector $V_{\bm\sigma}=(v_\tau)_{\tau\in \Delta^{\bm\sigma}_d}$ is given by
\begin{align*}
v_{\tau}&=i_\tau/x_\tau \text{  for } \tau \in \Delta_d,\\
v_{\bm\sigma} &= v_\sigma = i_\sigma/x_\sigma.
\end{align*}

\underline{Claim 1: $I_{\bm \sigma}$ satisfies KCL~\eqref{KCL}.}

For each $\eta\in \Delta_{d-1}$ and $j\in[n+1]$,
define $p(\eta,j)$ by
\begin{equation} \label{petaj}
p(\eta,j)=|\{j'  \in \eta \st j'<j \}|.
\end{equation}
If $(\partial_{\Delta^{\bm \sigma},d})_{\eta}$ denotes the row vector of $\partial_{\Delta^{\bm \sigma},d}$ indexed by $\eta$, then
\begin{equation} \label{deldeltaeta}
(\partial_{\Delta^{\bm \sigma},d})_{\eta} \,I_{\bm \sigma}
= \begin{cases}
\sum_{j \in[\ell_{d+1}] \setminus \eta }(-1)^{{p(\eta,j)}}i_{\eta \cup j}+(-1)^{p(\eta,\sigma \sm \eta)}i_{\bm \sigma} & \text{ if } \eta \subset \sigma,\\
\sum_{j \in[\ell_{d+1}] \setminus \eta  }(-1)^{{p(\eta,j)}}i_{\eta \cup j} & \text{ if } \eta \not\subset \sigma.
\end{cases}
\end{equation}
Claim~1 asserts that $(\partial_{\Delta^{\bm \sigma},d})_{\eta}\,I_{\bm \sigma}=0$. 
This follows immediately from~\eqref{current:zero} if $\eta\notin\shiftgen{\sigma}$.  In the remaining cases,
the idea of the proof will be to express $(\partial_{\Delta^{\bm \sigma},d})_{\eta}\,I_{\bm \sigma}$ as a telescoping series.
\smallskip

\underline{\textit{Case 1a:} $\eta\subset\sigma$}, say $\eta=\sigma\sm\ell_{s}=\{\ell_1<\cdots<\ell_{s-1}<\ell_{s+1}<\cdots<\ell_{d+1}$\}.  Observe that, for $j\not\in\eta$, the face $\eta\cup j$ belongs to $\shiftgen{\sigma}$ if and only if $j\leq\ell_s$.

Let $q \in[s]$.  If $\ell_{q-1}<j<\ell_{q}$, then
$\pi(\eta\cup j)=(1,2,\dots,q-1,s,q,q+1,\dots,s-1,s+1,\dots,d+1)$, which has $s-q$ inversions.
Putting $\tau=\eta\cup j$ in~\eqref{current:insigma} gives
\begin{equation} \label{ieta}
 i_{\eta \cup j} = (-1)^{s-q} \left(\prod_{r=1}^{q-1}{D_{{\ell_{r}}-1}} \right)\left(-x_j\right)\left(\prod_{r=q+1}^{s}{D_{\ell_{r-1}}}\right)\left(\prod_{r=s+1}^{d+1}{D_{\ell_{r}-1}}\right).
\end{equation}

Define $A(0)=0$ and
\begin{equation} \label{aq}
A(q)=A(q,s)=\left(\prod_{r=1}^{q}D_{\ell_{r}-1}\right)\left( \prod_{r=q+1}^{s}D_{\ell_{r-1}}\right)
\end{equation}
for $q>0$. Then summing~\eqref{ieta} over $\ell_{q-1}<j < \ell_q$, for each $q>0$, gives
\begin{align}
\sum_{ \ell_{q-1}<j < \ell_q} (-1)^{s-q} i_{\eta \cup j}
&= \left(\prod_{r=s+1}^{d+1}{D_{\ell_{r}-1}}\right) \left(\prod_{r=1}^{q-1}{D_{\ell_{r}-1}} \right)\left(\prod_{r=q+1}^{s}{D_{\ell_{r-1}}}\right)\left(-\sum_{\ell_{q-1}<j<\ell_q}{x_{j}}\right) \notag\\
&= \left(\prod_{r=s+1}^{d+1}{D_{\ell_{r}-1}}\right) \left(\prod_{r=1}^{q-1}{D_{\ell_{r}-1}} \right)\left(\prod_{r=q+1}^{s}{D_{\ell_{r-1}}}\right)\left(D_{\ell_{q-1}}-D_{\ell_q-1}\right)  \notag\\
&= \left(\prod_{r=s+1}^{d+1}{D_{\ell_{r}-1}}\right) \left(A(q-1)-A(q)\right). \label{telescope}
\end{align}
(Note that in the special case $q=1$, the terms $D_{\ell_0}=D_0$ and $A(0)$ vanish.)
Now summing~\eqref{telescope} over $q$ yields
\begin{align}
\sum_{q=1}^{s}\sum_{ \ell_{q-1}<j < \ell_q}
(-1)^{s-q} i_{\eta \cup j}
&= \sum_{q=1}^{s} \left(\prod_{r=s+1}^{d+1}{D_{\ell_{r}-1}}\right) \left(A(q-1)-A(q)\right) \notag
\\
&=
\left(\prod_{r=s+1}^{d+1}{D_{\ell_{r}-1}}\right) \sum_{q=1}^s  \left(A(q-1)-A(q)\right) \notag
\\
&=
-A(s)\prod_{r=s+1}^{d+1}{D_{\ell_{r}-1}} =-\prod_{r=1}^{d+1}{D_{\ell_{r}-1}}. \label{prodD}
\end{align}
Since $p(\eta,j)=q-1$ when $\ell_{q-1}<j<\ell_q$ and $p(\eta,\sigma\sm\eta)=s-1$, we obtain
\begin{align*}
(\partial_{\Delta^{\bm \sigma},d})_{\eta}\,I_{\bm \sigma}
&= \left(\sum_{q=1}^s\sum_{ \ell_{q-1}<j < \ell_q}(-1)^{q-1} i_{\eta \cup j}\right)+(-1)^{s-1}i_{\sigma}+(-1)^{s-1}i_{\bm\sigma}
	&&\text{(by~\eqref{deldeltaeta})}\\
&= (-1)^s\prod_{r=1}^{d+1}{D_{\ell_{r}-1}}+(-1)^{s-1}(i_{\sigma}+i_{\bm\sigma})
	&& \text{(by~\eqref{prodD})}\\
&= 0
	&& \text{(by~\eqref{current:sigma}).}\\
\end{align*}
\smallskip

\underline{\textit{Case 1b:} $\eta\in\shiftgen{\sigma}$ but $\eta \not\subseteq \sigma$}. Define
\[
M=\max\{q \in[d+1]:\ \ell_q\not\in\eta,\ \eta\cup\ell_q \in\shiftgen{\sigma} \}.
\]
The definition of a shifted complex implies that $M$ is well-defined (i.e., that at least one such $q$ exists).
Write $\eta \cup \ell_{M}=\{r_1<\dots<r_{d+1}\}$, so that $r_q\leq \ell_q$ for all $q$.  In fact $r_M=\ell_M$ (for otherwise $r_q=\ell_M$ for some $q>M$, but then $\eta\cup\ell_q\in\shiftgen{\sigma}$, contradicting the definition of $M$).
Define
\[
m=\max\{q \in[M-1]:\  r_q \neq \ell_q \},
\]
where by convention the maximum of the empty set is $0$. In particular, the definition of $m$ implies that $r_m<\ell_m$ (provided that $m>0$) and that $r_i=\ell_i$ for all $i$ with $m<i\leq M$ (the case $i=M$ was noted above).

In the remainder of Case~1b, we assume throughout that $j\notin\eta$. If $j > \ell_M$, then $i_{\eta \cup j}=0$ by the definition of $M$ and~\eqref{current:zero}. If on the other hand $j < \ell_m$, then the $(m+1)$st smallest element of $\eta\cup j$ is $\max(j,r_m)<\ell_m$, so $U^{\eta\cup j}_{m+1}=0$ by~\eqref{Utauq:1} and again $i_{\eta \cup j}=0$ by~\eqref{current:zero}.  Therefore, it is sufficient to consider $i_{\eta \cup j}$ for $j \in[\ell_{m}, \ell_{M}]$.

First, for a fixed $r \notin [m+1,M]$, observe that
$U_{r}^{\eta\cup j}$ is independent of the choice of $j\in[\ell_{m},\ell_{M}]\sm\eta=[\ell_{m},\ell_{M}]\sm\{\ell_{m+1},\dots,\ell_{M-1}\}$. So we may define
\[
U=(-1)^{p(\eta,\ell_{m})+1}(-1)^{\inv(\pi(\eta\cup \ell_m))} \prod_{r \notin [m+1,M]} U^{\eta \cup j}_r.
\]
It follows from ~\eqref{current:insigma} that 
\begin{align}
(-1)^{p(\eta,\ell_{m})} i_{\eta\cup \ell_{m}}
&= -(-1)^{p(\eta,\ell_{m})+1} (-1)^{\inv(\pi(\eta\cup \ell_{m}))}{\prod_{r=1}^{d+1}U^{\eta\cup \ell_{m}}_{r}} \notag\\
&=-U{\prod_{r=m+1}^M U^{\eta\cup \ell_{m}}_{r}}  \label{current:insigma:ell_m},
\end{align}
Second, for $j$ with $\ell_{m} < j \leq  \ell_{M}$, we can assume that $\ell_{r-1} < j < \ell_r$ for some $r \in [m+1,M]$ (where, for $r = M$, we also allow $j = \ell_r$). Then:
\begin{align*}
p(\eta,j)&=p(\eta,\ell_{m})+r-m-1 \quad\text{ and }\\
\inv(\pi(\eta\cup j))&=\inv(\pi(\eta \cup \ell_{m}))-(r-m),
\end{align*}
where $p(\eta,j)$ is defined as in~\eqref{petaj} and $\pi(\eta\cup j)$ is the permutation defined at the start of the proof. In particular, the numbers $p(\eta,j)+\inv(\pi(\eta\cup j))$ and $p(\eta,\ell_{m})+\inv(\pi(\eta \cup \ell_{m}))$ have opposite parity, so that
\[
U=(-1)^{p(\eta,j)}(-1)^{\inv(\pi(\eta\cup j))} \prod_{r \notin [m+1,M]}{U^{\eta \cup j}_{r}}.
\]
Therefore, ~\eqref{current:insigma} implies that
\begin{align}
(-1)^{p(\eta,j)} i_{\eta\cup j}
&= (-1)^{p(\eta,j)} (-1)^{\inv(\pi(\eta\cup j))}{\prod_{r=1}^{d+1}U^{\eta\cup j}_{r}} \notag\\
&= U{\prod_{r=m+1}^M U^{\eta\cup j}_{r}},  \label{current:insigma:2}
\end{align}
which we will use repeatedly in what follows. 

Define $B(0)=0$ and 
\[
B(q)=B(q,m,M)=\left(\prod_{r=m+1}^{q}{D_{\ell_{r}-1}} \prod_{r=q+1}^{M}{D_{\ell_{r-1}}}\right).
\]
for $q \in[m,M]$.

First, note that  $\eta\cup \ell_m=\{r_1<\cdots<r_m<\ell_m<\ell_{m+1}<\cdots<\ell_{M-1}<r_{M+1}<\cdots<r_{d+1}\}$.
Then by~\eqref{current:insigma:ell_m} and~\eqref{Utauq:2}
\begin{equation}\label{petajietaj:1}
(-1)^{p(\eta,\ell_m)} i_{\eta\cup \ell_m}
=- U\prod_{r=m+1}^{M} U^{\eta\cup \ell_m}_{r} =- U\prod_{r=m+1}^{M} D_{\ell_{r-1}}
= -U\cdot  B(m).
\end{equation}

Second, for each $q \in[m+1,M-1]$, summing~\eqref{current:insigma:2} for $\ell_{q-1}<j <\ell_q$ (where $\eta\cup j=\{r_1<\cdots<r_{q-1}<j<r_q<\cdots<r_{M-1}<\cdots\}$) and applying~\eqref{Utauq:2}--\eqref{Utauq:4} gives
\begin{align}
\sum_{ \ell_{q-1}<j <\ell_q}  (-1)^{p(\eta,j)} i_{\eta \cup j}
&= U\left(\prod_{r=m+1}^{q-1}{D_{\ell_{r}-1}} \prod_{r=q+1}^{M}{D_{\ell_{r-1}}}\right)\left(-\sum_{\ell_{q-1}<j<\ell_q}{x_{j}}\right)  \notag\\
&= U\left(\prod_{r=m+1}^{q-1}{D_{\ell_{r}-1}} \prod_{r=q+1}^{M}{D_{\ell_{r-1}}}\right)\left(D_{\ell_{q-1}}-D_{\ell_q-1}\right) \notag\\
&= U(B(q-1)-B(q)).  \label{petajietaj:2}
\end{align}
Third, summing~\eqref{current:insigma:2} for $\ell_{M-1}<j\leq\ell_M$ (where $\eta\cup j=\{r_1<\cdots<r_{M-1}<j<\cdots\}$) and applying \eqref{Utauq:3} and~\eqref{Utauq:4} gives
\begin{align}
\sum_{\ell_{M-1}<j \leq \ell_M}  (-1)^{p(\eta,j)} i_{\eta \cup j}
&= U\left(\prod_{r=m+1}^{M-1}{D_{j_{r}-1}}\right)\left(D_{\ell_M-1}-\sum_{\ell_{M-1}<j<\ell_M}{x_{j}}\right)  \notag\\
&=  U\left(\prod_{r=m+1}^{M-1}{D_{j_{r}-1}}\right)D_{\ell_{M-1}} \notag\\
&= U\cdot  B(M-1).  \label{petajietaj:3}
\end{align}
Combining~\eqref{petajietaj:1}, \eqref{petajietaj:2}, and~\eqref{petajietaj:3}, we conclude that
\begin{align*}
\sum_{ \ell_{m}  \leq j \leq \ell_{M}}{  (-1)^{p(\eta,j)} i_{\eta \cup j}} 
&=U\left( - B(m)+ \sum_{q=m+1}^{M-1}\left(B(q-1)-B(q)\right) +B(M-1)\right)
\\
&=0
\end{align*}
completing the proof of Claim~1.
\smallskip

\underline{Claim 2: $V=V_{\bm \sigma}$ satisfies KVL~\eqref{KVL}.}

That is, $V$ is orthogonal to $\ker\partial_{\Delta^{\bm\sigma},d}$, which is spanned by $[{\bm\sigma}]-[\sigma]$ together with the vectors $z_{\tau}=\partial_d[1\cup\tau]$, where $\tau$ ranges over $d$-faces of $\Delta$ not containing vertex~1.  Since $v_{\bm\sigma} = v_\sigma$, it suffices to show that $\langle V,z_\tau\rangle=0$ for every such~$\tau$.
\smallskip

\underline{\textit{Case 2a:} $\tau=\sigma$.}
Recall that we have assumed $\ell_1>1$.
Let $s\in[d+1]$; then $\pi(\sigma \setminus \ell_{s} \cup 1)$ is the cycle $(s,s-1,\dots,2,1)$, which has $s-1$ inversions, and \eqref{current:insigma} and~\eqref{Utauq:2}\dots\eqref{Utauq:4} give
\begin{align}
i_{\sigma \setminus \ell_{s} \cup 1}
&=(-1)^{s-1}(-x_1)\prod_{q=2}^{s}{D_{\ell_{q-1}}}\prod_{q=s+1}^{d+1}{D_{\ell_q-1}} \notag\\
&=(-1)^{s}x_1\prod_{q=2}^{s}{D_{\ell_{q-1}}}\prod_{q=s+1}^{d+1}{D_{\ell_q-1}}. \label{itau}
\end{align}
Moreover, using formula~\eqref{shifted-cycle} for $z_\sigma$ and applying \eqref{isigma} and~\eqref{itau}, we obtain
\begin{align*}
\langle V, z_\sigma \rangle
&= v_{\sigma} + \sum_{s=1}^{d+1}(-1)^s v_{\sigma \setminus \ell_{s} \cup 1}
\\
&=
\dfrac{\prod_{q=1}^{d+1}D_{\ell_q-1}-\prod_{q=1}^{d+1}D_{\ell_q}}{x_{\sigma}}
+ \sum_{s=1}^{d+1}\dfrac{x_1\prod_{q=2}^{s}{D_{\ell_{q-1}}}\prod_{q=s+1}^{d+1}{D_{\ell_q-1}}}{x_1 x_\sigma / x_{\ell_{s}}}
\\
&= \dfrac{1}{x_\sigma} \left[
\prod_{q=1}^{d+1}D_{\ell_q-1}-\prod_{q=1}^{d+1}D_{\ell_q} + \sum_{s=1}^{d+1}\left(\prod_{q=1}^{s-1}{D_{\ell_{q}}}\right)x_{\ell_{s}}\left(\prod_{q=s+1}^{d+1}D_{\ell_q-1}\right) \right]
\\
&= \dfrac{1}{x_\sigma} \left[
\prod_{q=1}^{d+1}D_{\ell_q-1}-\prod_{q=1}^{d+1}D_{\ell_q} + \sum_{s=1}^{d+1}\left(\prod_{q=1}^{s-1}{D_{\ell_{q}}}\right) (D_{\ell_s} - D_{\ell_s-1}) \left(\prod_{q=s+1}^{d+1}D_{\ell_q-1}\right) \right]
\\
&= 0
\end{align*}
since this expression is telescoping.
\smallskip

\underline{\textit{Case 2b:} $\tau \in \shiftgen{\sigma}\sm\{\sigma\}$ with $1 \notin \tau$.} Let $\tau=\{j_1<\cdots<j_{d+1}\}$. For convenience, let $\tau_{q}=\tau\setminus j_q \cup 1$. Then  the entry of $z_{\tau}$ corresponding to $\tau_q$ is $(-1)^{q}$, and so
\[
\langle V, z_\tau \rangle = v_\tau + \sum_{q=1}^{d+1} (-1)^q v_{\tau_q}.
\]
Define an integer 
\[
s=\max\{q: j_{r}=\ell_{r} \text{ for all } r \in[q]\}.
\]
Note that $0 \le s \le d$ by the assumption on $\tau$. If $q \geq s+2$, then the $(s+2)$nd smallest element of $\tau_q$ is $j_{s+1}<\ell_{s+1}$, so $U_{s+2}^{\tau_q}=0$ and $i_{\tau_{q}}=0$.  Therefore, the expression for $\langle V, z_\tau \rangle$ becomes
\begin{equation} \label{eq:claim2b}
\langle V, z_\tau \rangle = v_\tau + \sum_{q=1}^{s+1} (-1)^q v_{\tau_q},
\end{equation}
which we wish to show is zero.  Observe that for each $r \geq s+2$, the $d$-faces $\tau,\tau_1,\dots,\tau_{s+1}$ all have the same $r$th smallest element, namely $j_{r}$, so 
\[
U_{r}^{\tau}=U_{r}^{\tau_1}=\cdots=U_{r}^{\tau_{s+1}}.
\]
Denote this quantity by $U_{r}$, and let $U=\prod_{r \geq s+2}U_{r}$.
Then \eqref{current:insigma} gives
\begin{equation} \label{eq:claim2b:1}
(-1)^{\inv(\pi(\tau))}v_{\tau}=\dfrac{ \left(\prod_{r=1}^{s}D_{j_{r}-1} \right) \left( -x_{j_{s+1}}\right) U }{x_{\tau}}.
\end{equation}
For each $q \in[1,s+1]$, the permutation $\pi(\tau_q)$ is the product of the cycle $(q,q-1,\dots,2,1)$ with $\pi(\tau)$.  Therefore, for the case $1\leq q\leq s$ we have
\begin{align}
(-1)^{\inv(\pi(\tau))}\,v_{\tau_q}
&= (-1)^{\inv(\pi(\tau_q))+q-1}\,i_{\tau_q}/x_{\tau_q} \notag\\
&=-(-1)^q
\dfrac{ \left(-x_1 \right) \left(\prod_{r=2}^q D_{j_{r-1}}\right)\! 
\left(\prod_{r=q+1}^{s}D_{j_{r}-1} \right) \left( -x_{j_{s+1}}\right) U}
{x_{\tau_q}}
\notag\\
&= (-1)^q
\dfrac{\left(\prod_{r=2}^qD_{{j_{r-1}} }\right) x_{j_q} \left(\prod_{r=q+1}^{s}D_{j_{r}-1} \right) \left( -x_{j_{s+1}}\right) U}
{x_{\tau}},
\label{eq:claim2b:2}
\end{align}
while for the case $q=s+1$ we have
\begin{align}
(-1)^{\inv(\pi(\tau))}v_{\tau_{s+1}}
&=
(-1)^{\inv(\pi(\tau_{s+1}))+s}i_{\tau_{s+1}}/x_{\tau_{s+1}}\notag\\
&=
(-1)^{s} \frac{(-x_1)\left(\prod_{r=2}^{s+1} D_{j_{r-1}}\right) U}{x_{\tau_{s+1}}}\notag\\
&=
(-1)^{s} \frac{(-x_{j_{s+1}})\left(\prod_{r=2}^{s+1} D_{j_{r-1}}\right) U}{x_\tau}. \label{eq:claim2b:3}
\end{align}

Now substituting \eqref{eq:claim2b:1}--\eqref{eq:claim2b:3} into \eqref{eq:claim2b} gives

\begin{align*}
\frac{(-1)^{\inv(\pi(\tau))} x^\tau \langle V,z_\tau\rangle }{x_{j_{s+1}}U}
&= 
-\prod_{r=1}^{s} D_{j_{r}-1}
~-~
\sum_{q=1}^{s} 
\left(\prod_{r=2}^q D_{j_{r-1}}\right)
x_{j_q} 
\left(\prod_{r=q+1}^{s} D_{j_{r}-1} \right)
~+~
\prod_{r=2}^{s+1} D_{j_{r-1}}
\end{align*}
which vanishes by a calculation similar to that of Case~2a.
This completes the proof of Claim~2.
\smallskip

Finally, using Equation~\eqref{eq:Ratio}, we obtain the desired ratio of tree-numbers:
\[
\dfrac{\hat{k}_d(\Delta)}{\hat{k}_d{(\Psi)}}
= \dfrac{i_{\bm\sigma}}{i_{{\bm\sigma}}+i_{\sigma}}
= \prod_{q=1}^{d+1} \dfrac{D_{\ell_q}}{D_{\ell_q-1}}.\qedhere
\]
\end{proof}

\begin{ex}
Let $\sigma=235$ and $\Delta=\shiftgen{\sigma}$.  The currents $i_{\bm\sigma}$ and $i_{\sigma}$ are then
\begin{align*}
i_{\bm\sigma}&=(x_1+x_2)(x_1+x_2+x_3)(x_1+x_2+x_3+x_4+x_5) &&= D_2D_3D_5,\\
i_{\sigma}&=x_1(x_1+x_2)(x_1+x_2+x_3+x_4)-i_{\bm\sigma} &&= D_1D_2D_4-i_{\bm\sigma}.
\end{align*}
The currents of facets $\tau$ with $\dim(\tau \cap \sigma)=1$ are
\[
i_{135}=-x_1D_2D_4,\quad i_{125}=x_1D_2D_4,\quad i_{123}=-x_1D_2D_3,\quad  i_{234}=-D_1D_2x_4.
\]
The currents of facets $\tau$ with $\dim(\tau \cap \sigma)=0$ are 
\[
i_{124}=-x_1D_2x_4, \text{ and } i_{134}=x_1D_2x_4.
\]
For $\Psi=\Delta\setminus\sigma$, we have
\[
\dfrac{\hat{k}_2{(\Delta)}}{\hat{k}_2(\Psi)}= \dfrac{D_{2}D_{3}D_{5}}{D_{1}D_{2}D_{4}}=\dfrac{D_{3}D_{5}}{D_{1}D_{4}}.
\]
\end{ex}

\begin{ex}
Let $\sigma=246$ and $\Delta=\shiftgen{\sigma}$.  The currents $i_{\bm\sigma}$ and $i_{\sigma}$ are then
\begin{align*}
i_{\bm\sigma}&=(x_1+x_2)(x_1+x_2+x_3+x_4)(x_1+x_2+x_3+x_4+x_5+x_6)&&=D_2D_4D_6,\\
i_{\sigma}&=x_1(x_1+x_2+x_3)(x_1+x_2+x_3+x_4+x_5)-i_{\bm\sigma}&&=D_1D_3D_5-i_{\bm\sigma}.
\end{align*}
The currents of facets $\tau$ with $\dim(\tau \cap \sigma)=1$ are
\begin{align*}
i_{146} &=-x_1D_3D_5, & i_{126} &= x_1D_2D_5, & i_{124} &= -x_1D_2D_4,\\
i_{236} &=-D_1x_3D_5, & i_{234} &= D_1x_3D_4, & i_{245} &= -D_1D_3x_5.
\end{align*}
The currents of facets $\tau$ with $\dim(\tau \cap \sigma)=0$ are 
\begin{align*}
i_{136} &= x_1x_3D_5, & i_{235} &= D_1x_3x_5, & i_{134} &= -x_1x_3D_4,\\
i_{145} &= x_1D_3x_5, & i_{125} &= -x_1D_2x_5, & i_{123} &= 0.
\end{align*}
The other current is
\[
 i_{135} = -x_1x_3x_5.
\]
For $\Psi=\Delta\setminus\sigma$, we have
\[
\dfrac{\hat{k}_2{(\Delta)}}{\hat{k}_2(\Psi)}= \dfrac{D_{2}D_{4}D_{6}}{D_{1}D_{3}D_{5}}.
\]
\end{ex}

\begin{ex} Let $G$ be the threshold graph generated by an edge $\sigma=\ell_1 \ell_2$ with $1<\ell_1<\ell_2$. The current $i_{\bm\sigma}$ of ${\bm\sigma}$ and the current $i_{\sigma}$ of $\sigma$ are 
\[
i_{\bm\sigma}=D_{\ell_1}D_{\ell_2} \quad\text{and}\quad i_{\sigma}=D_{\ell_1-1}D_{\ell_2-1}-i_{\bm\sigma}.
\]
For $j_1,j_2 \in[\ell_2]$ with $j_1<j_2$ except for $(j_1,j_2)=(\ell_1,\ell_2)$, the current $i_{j_1j_2}$ is given by
\begin{equation} \label{eq:threshold_current}
i_{j_1j_2}=\begin{cases}
-x_{j_1}D_{\ell_2-1} & \text{ if } j_1<\ell_1, j_2=\ell_2,\\
-D_{\ell_1-1}x_{j_2} & \text{ if } j_1=\ell_1, j_2<\ell_2,\\
x_{j_1}D_{\ell_1} & \text{ if } j_1<\ell_1, j_2=\ell_1,\\
x_{j_1}x_{j_2} & \text{ if } j_1<\ell_1<j_2<\ell_2,\\
0 & \text{ otherwise.}
\end{cases}
\end{equation}
Thus
\[
\dfrac{\hat{k}_1{(G)}}{\hat{k}_1(G\setminus\sigma)}= \dfrac{D_{\ell_1}D_{\ell_2}}{D_{\ell_1-1}D_{\ell_2-1}}.
\]
In particular, let $\ell_2=\ell_1+1$ so that the second and fourth cases in ~\eqref{eq:threshold_current} do not happen. Then for $j_1,j_2 \in[\ell_2]$ with $j_1<j_2$ except for $(j_1,j_2)=(\ell_1,\ell_2)$, the current $i_{j_1j_2}$ is given by
\begin{equation*} 
i_{j_1j_2}=\begin{cases}
-x_{j_1}D_{\ell_1} & \text{ if } j_1<\ell_1, j_2=\ell_2,\\
x_{j_1}D_{\ell_1} & \text{ if } j_1<\ell_1, j_2=\ell_1,\\
0 & \text{ otherwise,}
\end{cases}
\end{equation*}
and
\[
\dfrac{\hat{k}_1{(G)}}{\hat{k}_1(G\setminus\sigma)}=\dfrac{D_{\ell_1}D_{\ell_2}}{D_{\ell_1-1}D_{\ell_2-1}}=\dfrac{D_{\ell_2}}{D_{\ell_1-1}}.
\]
\end{ex}

\subsection{Enumeration of simplicial spanning trees}
We now use simplicial effective resistance to enumerate the simplicial spanning trees of shifted complexes.  We start with some definitions that do not require shiftedness, but which will be useful in the proof.

\begin{defn}\label{def:block}
Let $\sigma=\{v_1 < \cdots < v_k\}$ be an ordered set of distinct positive integers.  A {\bf block} of $\sigma$ is a subset $B = \{v_p, v_{p+1}, \ldots, v_{q-1}, v_q\}$, where for $p \leq i < q$, we have $v_{i+1} = v_i + 1$, and where $v_p - 1, v_q + 1 \not\in \sigma$.  
In this case, we set $s_B = v_p$ and $t_B = v_q$ to be the starting and ending vertices, respectively, of the block $B$.  In other words, $B$ is a maximal subset of consecutive integers in $\sigma$, starting at $s_B$ and ending at $t_B$.  We may evidently partition $\sigma$ into disjoint blocks, $\sigma = \cup_{B \in \B(\sigma)} B$, where $\B(\sigma)$ is the set of blocks in $\sigma$.
\end{defn}

\begin{defn}\label{def:label}
Observe that each covering relation in Gale order is of the form $\lambda\coveredby\mu$, where
\[\lambda = \{a_1,\ldots,a_k\}, \qquad \mu = \{a_1,\ldots,a_{j-1},a_j+1,a_{j+1},\ldots,a_k\}.\]
We define the \textbf{label} of this covering relation as $\ell(\lambda\coveredby\mu) = a_j$.  
\end{defn}

Note that the covering relation is possible only if $a_j+1 < a_{j+1}$, 
so $a_j$ must be the end of a block of $\lambda$ and $a_j+1$ must be the start of a block of $\mu$.  Therefore, the covering relations involving $\sigma$ are, for each $B \in \B(\sigma)$:
\begin{itemize}
	\item $\sigma \coveredby \sigma - \{t_B\} \dju \{t_B + 1\}$ with label $t_B$; and
	\item $\sigma - \{s_B\} \dju \{s_B - 1\} \coveredby \sigma$ with label $s_B - 1$, if $s_B > 1$.
\end{itemize}
In particular, if $1 \not\in \sigma$, then $\sigma$ covers exactly $|B(\sigma)|$ sets and is covered by exactly $|B(\sigma)|$ sets.

Now assume that $\Delta$ is a $d$-dimensional shifted complex.  Recall the definitions of $\Gamma$ and $\Lambda$ from the beginning of this section.

\begin{defn}
When  $\sigma \in \Gamma$ and $\tau \not\in\Gamma$ are facets of $\Delta$, and $\sigma\coveredby\tau$, then~\cite[Sec.~4.3]{DKM} calls $(\sigma,\tau)$ a \textbf{critical pair} of $\Gamma$, and defines the {\bf signature} of that critical pair to be $(S,T)$ where $S$ is unimportant to us (it is useful for a finer weighting than we consider here; see \cite[Secs.~6--9]{DKM1}), and $T = \{2,\ldots,\ell(\sigma\coveredby\tau)\}$.  Let $\Sig(\Gamma)$ denote the set of all signatures of $\Gamma$.  
\end{defn}

\begin{defn}
For any simplicial complex $\Sigma$ and vertex $v$, denote by $\deg_{\Sigma}(v)$ the {\bf degree} of $v$ in $\Sigma$, the number of facets of $\Sigma$ containing $v$. 
\end{defn}

We can now state and prove the main theorem on spanning trees of shifted complexes.  For a nonempty set $T\subseteq\Nn$, let $\max T$ denote its maximum element.

\begin{thm}\label{thm:shifted}
\[
	\hat{k}_d(\Delta) = x_1^{\abs{\Lambda_{d-1}}}\prod_i x_i^{\deg_\Lambda(i)}
				\prod_{(S,T) \in \Sig(\Gamma)} \frac{D_{\max T}}{D_1}.
\]
\end{thm}
This formula was proved in~\cite[eqn.~25]{DKM} using a different recursive property of shifted complexes (namely, that they are near-cones).  Here we give a proof using simplicial effective resistance.

\begin{proof}
We start building $\Delta$ from its spanning tree $1 \ast \Lambda$.  Since it is a spanning tree, 
$\hat{k}_d(1 \ast \Lambda)$ is simply the product of the weights of its facets, so
\begin{equation}\label{1Lambda}
	\hat{k}_d(1 \ast \Lambda)
		= \prod_{\sigma \in 1 \ast \Lambda_{d-1}}\prod_{i \in \sigma} x_i
		= \prod_{\rho \in \Lambda_{d-1}}x_1 \prod_{i \in \rho} x_i
		= x_1^{\abs{\Lambda_{d-1}}}\prod_i x_i^{\deg_\Lambda(i)}.
\end{equation}

Now we turn our attention to the remainder of $\Delta$.  We will add one facet of $\Gamma$ at a time in lexicographic order.  For each facet $\sigma \in \Gamma$, define $\Delta_{< \sigma}$ to be the complex generated by all of $1 \ast \Lambda$ and the facets added before $\sigma$ (the facets lexicographically earlier than $\sigma$).  
Say $\sigma=\{v_1 < \cdots < v_{d+1}\}$.  Since $\sigma \in \Gamma$, we know $v_1 > 1$.  
Therefore, we can apply Theorem~\ref{thm:SC-ratio} to compute
\begin{equation}\label{eq:quotient}
	\frac{\hat{k}_d(\Delta_{< \sigma} \cup \sigma)}{\hat{k}_d(\Delta_{< \sigma})} 
		= \prod_{1 \leq i \leq d+1} \frac{D_{v_i}}{D_{v_i-1}} 
		= \prod_{B \in \B(\sigma)} \prod_{s_B \leq v_i \leq t_B}\frac{D_{v_i}}{D_{v_i-1}}
		= \prod_{B \in \B(\sigma)} \frac{D_{t_B}}{D_{s_B-1}}
		= \frac{\prod_{\tau \covers \sigma}D_{\ell(\sigma \coveredby \tau)}}{\prod_{\rho \coveredby \sigma}D_{\ell(\rho \coveredby \sigma)}}.
\end{equation}

By applying Equation~\eqref{eq:quotient} recursively,
\begin{align}
\hat{k}_d(\Delta)
&= \hat{k}_d(1 \ast \Lambda)\prod_{\sigma \in \Gamma_d}
	\frac{\hat{k}_d(\Delta_{< \sigma} \cup \sigma)}{\hat{k}_d(\Delta_{< \sigma})} 
	= \hat{k}_d(1 \ast \Lambda)\prod_{\sigma \in \Gamma_d}\frac{\prod_{\tau \covers \sigma}D_{\ell(\sigma \coveredby \tau)}}{\prod_{\rho \coveredby \sigma}D_{\ell(\rho \coveredby \sigma)}}\notag\\
	&= \hat{k}_d(1 \ast \Lambda)\frac{\prod_{\sigma \in \Gamma_d, \tau \not\in\Gamma_d, \sigma\coveredby\tau}D_{\ell(\sigma\coveredby\tau)}}
			{\prod_{\sigma \in \Gamma_d, \rho \not\in\Gamma_d, \rho\coveredby\sigma}D_{\ell(\rho\coveredby\sigma})}.
\label{eq:cancel}
\end{align}
The final equation follows from cancellation: every pair $\lambda \coveredby \mu$ where $\lambda$ and $\mu$ are both faces in $\Gamma_d$ appears once in the numerator when $\sigma = \lambda$, and once in the denominator when $\sigma = \mu$.  So the only terms that remain in the numerator are those where $\sigma \coveredby \tau$, but $\tau \not\in \Gamma_d$, and the only terms that remain in the denominator are those where $\rho \coveredby \sigma$, but $\rho \not\in \Gamma_d$.

Furthermore, the only way that the conditions $\sigma \in \Gamma_d$, $\rho \not\in\Gamma_d$, and $\rho\coveredby\sigma$ can all hold is when $1 \in \rho$, so $\ell(\rho\coveredby\sigma) = 1$, and so all the factors in the denominator in~\eqref{eq:cancel} are $D_1$.  And since every face in $\Gamma_d$ has as many covering relations going up (contributing a factor to the numerator) as going down (contributing a factor to the denominator), and each cancellation removes exactly one factor from the numerator and one factor from the denominator, the number of $D_1$ factors in the denominator equals the total number of factors in the numerator.

Meanwhile, the condition $\sigma \in \Gamma_d, \tau \not\in\Gamma_d, \sigma\coveredby\tau$ in the numerator is precisely the condition for $(\sigma, \tau)$ to be a critical pair of $\Gamma_d$, and the maximum element of $T$ in the signature $(S,T)$ of $(\sigma, \tau)$ is $\ell(\sigma \coveredby \tau)$.  We can now rewrite $\hat{k}_d(\Delta)$ as 
\begin{equation}\label{Gamma}
	\hat{k}_d(\Delta) 
		= \hat{k}_d(1 \ast \Lambda)\frac{\prod_{\sigma \in \Gamma_d, \tau \not\in\Gamma_d, \sigma\coveredby\tau}D_{\ell(\sigma\coveredby\tau)}}
		{\prod_{\sigma \in \Gamma_d, \rho \not\in\Gamma_d, \rho\coveredby\sigma}D_{\ell(\rho\coveredby\sigma})}
		= \hat{k}_d(1 \ast \Lambda)\prod_{(S,T) \in \Sig(\Gamma)} \frac{D_{\max T}}{D_1}.
\end{equation}
The result now follows by combining Equations~\eqref{1Lambda} and~\eqref{Gamma}.
\end{proof}

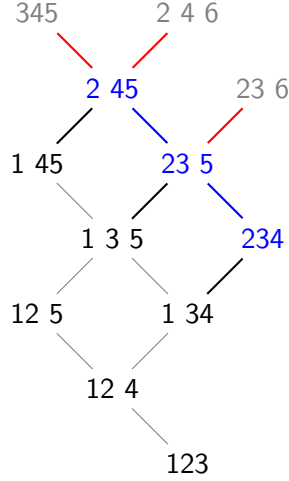
\begin{figure}
\begin{center}
\begin{tikzpicture}
	\node (123) at (0,0) {\sf123};
	\node (124) at (-1,1) {\sf12 4};
	\node (125) at (-2,2) {\sf12 5};
	\node (134) at (0,2) {\sf1 34};
	\node (135) at (-1,3) {\sf1 3 5};
	\node (145) at (-2,4) {\sf1 45};
	
	\node [blue,thick] (234) at (1,3) {\sf234};
	\node [blue,thick] (235) at (0,4) {\sf23 5};
	\node [blue,thick] (245) at (-1,5) {\sf2 45};
	
	\node [gray,thick] (236) at (1,5) {\sf23 6};
	\node [gray,thick] (345) at (-2,6) {\sf345};
	\node [gray,thick] (246) at (0,6) {\sf2 4 6};

	\draw [gray] (123) -- (124);
	\draw [gray] (124) -- (125);
	\draw [gray] (124) -- (134);
	\draw [gray] (134) -- (135);
	\draw [gray] (125) -- (135);
	\draw [gray] (135) -- (145);
	
	\draw [black,thick] (134) -- (234);
	\draw [black,thick] (135) -- (235);
	\draw [black,thick] (145) -- (245);
	
	\draw [blue,thick] (234) -- (235);
	\draw [blue,thick] (235) -- (245);
	
	\draw [red,thick] (235) -- (236);
	\draw [red,thick] (245) -- (345);
	\draw [red,thick] (245) -- (246);
\end{tikzpicture}
\end{center}
\caption{Componentwise partial order of facets of a shifted complex.}\label{fig:cptwise}
\end{figure}

\begin{ex}
Consider the shifted complex $\Delta$ whose facets are the Gale order ideal $\shiftgen{245}$.  The Gale order on the facets, and also including some of the non-faces that cover them in Gale order, is shown in Figure~\ref{fig:cptwise}.  Faces in $1 \ast \Lambda$ are colored black, faces in $\Gamma$ are colored blue, and the faces outside $\Delta$ that cover faces in $\Gamma$ are colored gray.  Edges entirely in $1 \ast \Lambda$ are colored gray, edges between $1 \ast \Lambda$ and $\Gamma$ are colored black, edges entirely in $\Gamma$ are colored blue, and edges between faces in $\Gamma$ and faces outside $\Delta$ are colored red.  The diagram does not include faces outside of $\Delta$ that cover only faces in $1 \ast \Lambda$, since they do not enter our calculations.

Let us first illustrate Definitions~\ref{def:block} and~\ref{def:label} with $\sigma = 235$, whose blocks are $23$ and $5$.  We will emphasize the block structure by writing $\sigma$ as $23\ 5$, and similarly for other faces in $\Delta$.  The faces covering $\sigma$ are $2\ 45$ with label $3$ and $23\ 6$ with label $5$.  The faces that $\sigma$ covers are $1\ 3\ 5$ with label $1$ and $234$ with label $4$.

We next illustrate Equation~\eqref{eq:quotient}.  When we add $\sigma = 23\ 5$, the quotient is
\[
\frac{D_2 D_3\ D_5}{D_1 D_2\ D_4}
= \frac{D_3}{D_1}\frac{D_5}{D_4}
= \frac{D_{\ell(23\ 5 \coveredby 2\ 45)}D_{\ell(23\ 5 \coveredby 23\ 6)}}{D_{\ell(1\ 3\ 5 \coveredby 23\ 5)}D_{\ell(234 \coveredby 23\ 5)}}.
\]
But when we multiply together all such quotients, as in Equation~\eqref{eq:cancel},  some of the terms above will be cancelled.  In particular, $D_{\ell(23\ 5 \coveredby 2\ 45)}$ will appear in the denominator when $2\ 45$ is added, and $D_{\ell(234 \coveredby 23\ 5)}$ was in the numerator when $234$ was added.  All told, all the labels on edges in blue, between faces that are both in $\Gamma$, will contribute cancelling factors, in the numerator when the lower face on the edge is added, but in the denominator when the upper face on the edge is added.  This will leave only, in the numerator, the factors corresponding to labels on red edges, which here are 
\[
	D_{\ell(23\ 5 \coveredby 23\ 6)} D_{\ell(2\ 45 \coveredby 2\ 4\ 6)} 
	D_{\ell(2\ 45 \coveredby 345)} 
	= D_5^2 D_2,
\] 
and, in the denominator, the factors corresponding to labels on black edges, which here are
\[
	D_{\ell(1\ 34 \coveredby 234)} D_{\ell(1\ 3\ 5 \coveredby 23\ 5)} 
	D_{\ell(1\ 45 \coveredby 2\ 45)} 
	= D_1^3. 
\]
The red edges correspond to signatures, and the black edges necessarily are labeled by $1$.

Note how we started with 5 edges going up and 5 edges going down (1 up and 1 down from $234$ and 2 up and 2 down each from $23\ 5$ and $2\ 45$) from the faces in $\Gamma$, but 2 pairs were cancelled (one cancelling pair for each blue edge), leaving 3 edges going up and 3 edges going down.  This then corresponds to having 3 factors in the numerator and 3 factors in the denominator.
\end{ex}

\bibliographystyle{amsalpha}
\bibliography{biblio}
\end{document}